\documentclass{amsart}

\usepackage{graphicx}
\usepackage{framed}

\usepackage{dsfont}
\usepackage{mathrsfs}

\usepackage[utf8]{inputenc}

\usepackage[all]{xy}
\usepackage{pb-diagram, pb-xy}

\usepackage{tikz}
\usepackage{pgffor}
\usetikzlibrary{arrows,backgrounds}
\usetikzlibrary{calc}

\usepackage{amssymb}
\usepackage{amsthm}

\usepackage{hyperref}

\usepackage{calc}

\newtheoremstyle{plain}         % name
  {1\baselineskip}              % Space above, empty = `usual value'
  {1\baselineskip}              % Space below
  {\slshape}                    % Body font
  {}                            % Indent amount (empty = no indent, \parindent = para indent)
  {\bfseries}                   % Thm head font
  {.}                           % Punctuation after thm head
  {.5em}                        % Space after thm head: " " = normal interword space;
  {\thmnumber{#2 }\thmname{#1}\thmnote{ \normalfont\textsc{(#3)}}} % Thm head spec

\theoremstyle{plain}
\newtheorem{thm}{Theorem}[section]

\newtheorem{prop}[thm]{Proposition}
\newtheorem{cor}[thm]{Corollary}
\newtheorem{lem}[thm]{Lemma}

\theoremstyle{definition}
\newtheorem{example}[thm]{Example}
\newtheorem{remark}[thm]{Remark}

\newtheoremstyle{special}       % name
  {1\baselineskip}              % Space above, empty = `usual value'
  {1\baselineskip}              % Space below
  {\slshape}                    % Body font
  {}                            % Indent amount (empty = no indent, \parindent = para indent)
  {\bfseries}                   % Thm head font
  {.}                           % Punctuation after thm head
  {.5em}                        % Space after thm head: " " = normal interword space;
  {\thmnumber{#2 }\thmnote{#3}} % Thm head spec

\theoremstyle{special}

\newtheoremstyle{note}          % name
  {1\baselineskip}              % Space above, empty = `usual value'
  {1\baselineskip}              % Space below
  {\normalfont}                 % Body font
  {}                            % Indent amount (empty = no indent, \parindent = para indent)
  {\bfseries}                   % Thm head font
  {:}                           % Punctuation after thm head
  {.5em}                        % Space after thm head: " " = normal interword space;
  {\thmnumber{#2 }\thmname{#1}\thmnote{#3}}    % Thm head spec

\theoremstyle{note}

\newtheorem{bem*}[thm]{Bemerkung}

\def\N{\mathbb{N}}
\def\Z{\mathbb{Z}}  %ersetzt durch \Z

\def\1{\mathds{1}}

\newcommand{\A}{\mathcal{A}}
\newcommand{\NN}{\mathcal{N}}

\def\A{\mathbb{A}}

\newcommand{\bbZ}{{\mathbb Z}}

\newcommand{\calR}{{\mathcal R}}

\newcommand{\one}{{\mathbf{1}}}

\newcommand{\g}{\mathfrak{g}}

\newcommand{\lift}{\operatorname{lift}}
\def\uRep{\underline{\operatorname{Re}}\!\operatorname{p}} 
\usepackage{xypic,dsfont}

\hypersetup{
    colorlinks,
    citecolor=red,
    filecolor=black,
    linkcolor=blue,
    urlcolor=black
}

\begin{document}

\sloppy

\thispagestyle{empty}
\title{Mixed tensors of the General Linear Supergroup}
\author{Thorsten Heidersdorf} 
\address{T.H.: Department of Mathematics, The Ohio State University}
\email{heidersdorf.thorsten@gmail.com}

\date{}

\begin{abstract} We describe the image of the canonical tensor functor from Deligne's interpolating category $\uRep(GL_{m-n})$ to $Rep(GL(m|n))$ attached to the standard representation. This implies explicit tensor product decompositions between any two projective modules and any two Kostant modules of $GL(m|n)$, covering the decomposition between any two irreducible $GL(m|1)$-representations. We also obtain character and dimension formulas. For $m>n$ we classify the mixed tensors with non-vanishing superdimension. For $m=n$ we characterize the maximally atypical mixed tensors and show some applications regarding tensor products.
\end{abstract}


\maketitle

\section{Introduction}

In this article we describe the indecomposable summands of the mixed tensor space $V^{\otimes r} \otimes (V^{\vee})^{\otimes s}$ ($r,s \in \N)$ where $V = k^{m|n}$ is the standard representation of the General Linear Supergroup $GL(m|n)$ ($m \geq n$) over an algebraically closed field $k$ of characteristic zero. Such a summand is called a mixed tensor. These results imply decomposition laws for the tensor product between mixed tensors, character and dimension formulas and give us estimates about composition factors and Loewy lengths in tensor products between maximal atypical irreducible modules.

In the category of finite-dimensional algebraic representations $Rep(GL(m|n))$ the decomposition of the tensor product of two irreducible modules is known for a very small class of representations, the direct summands occurring in a tensor power of the standard representation $V \simeq k^{m|n}$. The tensor product product decomposition is given by the Littlewood Richardson Rule. By \cite{Sergeev} \cite{Berele-Regev} the tensor space $V^{\otimes r}$ is completely reducible and the irreducible representations obtained in this way - the covariant representations - can be parametrized by $(m|n)$-hook partitions. It turns out that these representations form only a very small subset of the irreducible $GL(m|n)$-representations. In this article we look at the larger space of mixed tensors $V^{\otimes r} \otimes (V^{\vee})^{\otimes s}$, $r,s \in \N$. Since $V^{\vee}$ is not unitary, this space is not fully reducible. The direct summands can be described via the Khovanov algebras of Brundan and Stroppel \cite{Brundan-Stroppel-5} and their tensor product decomposition can be understood using Deligne's interpolating categories. In \cite{Deligne-interpolation} Deligne constructed for any $\delta \in k$ a karoubian rigid symmetric monoidal category $\uRep(GL_{\delta})$ which interpolates the classical representation categories $Rep(GL(n))$ in the sense that for $\delta=n \in \N$ we have an equivalence of tensor categories $\uRep(GL_{n})/\NN \to Rep(GL(n))$ where $\NN$ denotes the tensor ideal of negligible morphisms \cite{Andre-Kahn}. These interpolating categories possess a distinguished element of dimension $\delta$ which we call the standard representation $st$. Deligne's family of tensor categories are the universal tensor categories on a dualisable object of dimension $\delta$ in the sense of the universal property \ref{Deligne-interpolation}.

In particular for $m-n \in \N_{\geq 0}$ we have two tensor functors starting from the Deligne category $\uRep(GL_{m-n})$: One into $Rep(GL(m-n))$, the other one into $Rep(GL(m|n))$ (both determined by the choice of the standard representations $V = k^{m-n}$ respectively $V = k^{m|n}$). The tensor product decomposition in Deligne's category has been determined by Comes and Wilson \cite{Comes-Wilson}. If we are then able to understand the functor $F_{m|n}: \uRep(GL_{m-n}) \to Rep(GL(m|n))$, $st \mapsto V$, we will be able to decompose tensor products in its image.  Comes and Wilson also determine the kernel of the functor $F_{m|n}$ and show that its image is the space of mixed tensors $T$: The full subcategory of $Rep(GL(m|n))$ of objects which are direct summands in a tensor product $V^{\otimes r} \otimes (V^{\vee})^{\otimes s}$ for some $r,s \in \N$. However Comes and Wilson do not describe the image $F_{m|n}(X)$ of an individual object $X$. 

%\bigskip

\subsection{Main results} The space of mixed tensors has also been studied by Brundan and Stroppel \cite{Brundan-Stroppel-5}.  In both approaches the indecomposable mixed tensors $R(\lambda)$ are described by certain pairs $\lambda = (\lambda^L,\lambda^R)$ of partitions, so-called $(m|n)$-cross bipartitions. The advantage of Brundan and Stroppels results is that they permit to analyze the Loewy structures of the mixed tensors and gives conditions on their highest weights. This allows to identify the image of an element under the tensor functor  $\uRep(GL_{m-n}) \to Rep(GL(m|n)$. In section \ref{indecomposable} we define two invariants $d(\lambda)$ and $k(\lambda)$ of a bipartition.

\begin{thm} (see \ref{loewy-length}, \ref{projective-mixed-tensors-1}, \ref{projective-mixed-tensors-2}, \ref{sec:theta}) The Loewy length of a mixed tensor $R(\lambda)$ is $2d(\lambda) + 1$. In particular it is irreducible if and only if $d(\lambda) = 0$ and projective if and only if $k(\lambda) = n$. Every projective module is a mixed tensor. We have an explicit bijection $\theta_n$ between the bipartitions with $k(\lambda) = n$ and the projective covers of irreducible modules. Similarly we have an explicit bijection $\theta_0$ between the bipartitions with $d(\lambda) = 0$ and the irreducible mixed tensors.
\end{thm}

Hence we obtain a decomposition law for tensor products between projective representations of $GL(m|n)$. Another application is an easy description of the dual of an irreducible representation in section \ref{duals}. We study the effect of the cohomological tensor functors $DS: Rep(GL(m|n)) \to Rep(GL(m-1|n-1))$ of \cite{Duflo-Serganova} \cite{Heidersdorf-Weissauer-tensor} on mixed tensors in section \ref{DS}. We show that the kernel of $DS: Rep(GL(m|n)) \to Rep(GL(m-1|n-1)$  consists of the projective representations and that $DS(R(\lambda)) = R(\lambda)$ if $k(\lambda) < n$. Our next aim is characterize which irreducible representations occur as mixed tensors. For that recall that an irreducible representation $L(\mu)$ is a Kostant module if and only if the Kazhdan-Lusztig polynomials $p_{\lambda,\mu}(q)$ are monomials for all $\lambda \leq \mu$. Using that every representation splits into a direct sum of typical representations after repeated application of $DS$ and that every typical representation is a mixed tensor, we obtain the next theorem \ref{Kostant-existence}  \ref{properties-kostant} . 

\begin{thm} Every irreducible mixed tensor $L(\mu)$ is a Kostant module. In particular $p_{\lambda,\mu}(q) =
q^{l(\lambda,\mu)}$ for all $\lambda \leq \mu$ (and some combinatorially defined number $l(\lambda,\mu)$) and therefore $\sum_{i \geq 0} dim Ext^i (K(\lambda),L(\mu)) \leq 1$ for all $\lambda \in X^+$. Conversely, every Kostant module is a (known) Berezin twist of an irreducible mixed tensor.
\end{thm}

As an application we obtain a formula for the tensor product decomposition between two Kostant modules. Since every singly atypical irreducible module is a Kostant-module and every typical module is projective, the last two theorems solve the problem of decomposing tensor products between any two irreducible $GL(m|1)$-representations. Our results imply (theorem \ref{tensor-equivalence}) also the equivalence of tensor categories \[ T/\NN \simeq Rep(GL(m-n)) \] which is used in a crucial way in \cite{Heidersdorf-semisimple}. Since we have character and dimension formulas for every mixed tensor by \cite[Theorem 8.5.2]{Comes-Wilson}, we get character and dimension formulas for Kostant modules and projective modules in section \ref{character-dimension}. %These formulas are much more explicit then the general formulas in \cite{Su-Zhang}.

In the remaining sections \ref{maximal-atypical-m=n} - \ref{sec:tensor-products} we study maximal atypical mixed tensors for $m=n$. In the $m=n$ case no nontrivial maximal atypical irreducible representation of $Rep(GL(n|n))$ is in the image of $F_{n|n}: \uRep(GL_0) \to Rep(GL(n|n))$. We characterise the maximal atypical mixed tensors in proposition \ref{max-atypical} and show in theorem \ref{thm:existence} that every maximal atypical representation can be realised (up to a Berezin-shift) as the socle respectively highest weight constituent in a non-projective mixed tensor. In section \ref{symmetric-powers} we study the class of the smallest maximally atypical mixed tensors (the ones of minimal Loewy length) which we call the symmetric powers $\A_{S^i}, i \in \N$. We apply the results of the previous sections in section \ref{sec:tensor-products} to get some general estimates about composition factors and projective modules in tensor products of irreducible representations. In particular we obtain estimates about the maximal Loewy length of a summand in a tensor product of two maximally atypical irreducible representations. For applications of these results we refer the reader to \cite{Heidersdorf-Weissauer-gl-2-2} \cite{Heidersdorf-Weissauer-Tannaka}. 

%Since we have a nice character formula and a nice dimension formula for every mixed tensor by \cite{Comes-Wilson}, thm 8.5.2, we get character and dimension formulas for any Kostant modules and any projective modules. In the $m=n$ case no maximal atypical irreducible representation of $Rep(GL(n|n))$ is in the image of $F_{n|n}: Rep(GL_0) \to Rep(GL(n|n))$. In the third part we study the class of the \textit{smallest} maximally atypical tensors (the ones of minimal Loewy length) which we call the symmetric powers $\A_{S^i}, i \in \N$.  We derive a closed formula for their tensor products $\A_{S^i} \otimes \A_{S^j}$ (a generalized Pieri rule). As all mixed tensors these have a simple socle which we denote by $S^{i-1}$. One might hope to infer back from the $\A_{S^i} \otimes \A_{S^j}$-tensor product to the $S^{i-1} \otimes S^{j-1}$-tensor product. This is indeed true as it will enable us in \cite{Heidersdorf-Weissauer-Tannaka} to decompose tensor products of the form $S^i \otimes S^j$. We remark that $\A = \A_{S^1}$ is the adjoint representation. We end the article with some general remarks about composition factors and projective module in tensor products in section \ref{sec:remarks}.

\subsection{Acknowledgements} I would like to thank the referee for many helpful suggestions.

%------------------------------------------------

%-------------------------------------------------

\section{Representations of $GL(m|n)$}\label{sec:repr}
%\addcontentsline{toc}{section}{Preliminaries}

{\it Representations}. Let $k$ be an algebraically closed field of characteristic $0$. For the linear supergroup $G=GL(m|n)$, $m \geq n$, over $k$
let $F$ be the category of 
super representations $\rho$ of $GL(m\vert n)$ on finite dimensional super vectorspaces over $k$. We assume throughout the article that $m \geq n$ for simplicity of notation. This is not a restriction since $Gl(m|n) \simeq Gl(n|m)$ for any $m,n$.
The morphisms in the category $F$ are the $G$-linear maps $f:V \to W$ between super representations, 
where we allow even and odd morphisms with respect to the gradings on $V$ and $W$. Let $F^{ev}$ be the subcategory of $F$ with the same objects as $F$ and $Hom_{F^{ev}}(X,Y)=Hom_F(X,Y)_{0}$. We often write $F^{ev} = T_{m|n}$.

{\it The category ${\calR}$}. Fix the morphism $\varepsilon: \mathbb Z/2\mathbb Z \to G_0=GL(m)\times GL(n)$ which maps $-1$ to the element 
$diag(E_m,-E_n)\in GL(m)\times GL(n)$ denoted $\epsilon_{mn}$. Note that
$Ad(\epsilon_{mn})$ induces the parity morphism on the Lie superalgebra $\mathfrak{gl}(m|n)$ of $G$. 
We define the abelian subcategory
$\calR_{m|n} = \calR$ of $T_{m|n}$ as the full subcategory of all objects $(V,\rho)$ in $T_{m|n}$  
with the property $  p_V = \rho(\epsilon_{mn})$; here $\rho$ denotes the underlying homomorphism $\rho: GL(m)\times GL(n) 
\to GL(V)$ of algebraic groups over $k$ and $p_V$ the parity automorphism.
The subcategory ${\calR}$ is stable under the dualities ${}^\vee$ (the ordinary dual) and $^*$ (the graded dual \cite{Germoni-sl}). 
For $G=GL(n|n)$ we usually write ${\calR}_n$ instead of $\calR$ or $\calR_{nn}$.
The abelian category $T_{m|n}$ decomposes as $T_{m|n} = {\calR}_{mn} \oplus \Pi ({\calR}_{mn})$ by\cite{Brundan-Kazhdan}, Cor. 4.44 where $\Pi$ denotes the parity shift functor.

The irreducible representations in $\calR$ are parametrized by the (integral dominant) highest weights $X^+$ \[ \lambda = \sum_{i = 1}^m \lambda_i \epsilon_i + \sum_{j=m+1}^{m+n}  \lambda_j \delta_j = (\lambda_1,\ldots,\lambda_m| \lambda_{m+1}, \ldots, \lambda_{m+n})\] with respect to the choice of the standard Borel group and the usual basis elements $\epsilon_i, \delta_j$ \cite{Germoni-sl}. Here $\lambda_1 \geq \ldots \geq \lambda_{m}$ and $\lambda_{m+1} \geq \ldots \geq \lambda_{m+n}$ are integers and every $\lambda \in \Z^{m+n}$ with these properties parametrises a highest weight of an irreducible representation. The irreducible representations in $T_{m|n}$ are given by the set $ \{L(\lambda), \Pi L(\lambda) \ | \ \lambda \in X^+ \}$.  We denote by $K(\lambda)$ the Kac-module of the weight $\lambda$ and by $P(\lambda)$ the projective cover of the irreducible representation $L(\lambda)$.

\textit{Atypicality.} If $K(\lambda)$ is irreducible the weight $\lambda$ is called typical. If not, $\lambda$ is called atypical. $K(\lambda)$ is irreducible if and only if $K(\lambda)$ is projective \cite{Kac-Rep}. The atypicality of a weight can be measured by a number between $0$ and $n$. If the atypicality is $n$, we say the weight is maximal atypical. Examples are the trivial module $\one$ and the standard representation $V = k^{m|n}$ of highest weight $\lambda = (1,\ldots,0|0,\ldots,0)$ for $m\neq n$. Another example is the Berezin determinant  $ B = Ber = L(1,\ldots,1 \ | \ -1,\ldots,-1)$ of dimension $1$. The abelian categories $T_{m|n}$ and $\calR$ decompose into blocks and the degree of atypicality is a block-invariant. %Hence we can define the degree of atypicality of an arbitrary indecomposable module to be the degree of atypicality of its composition factors. The full subcategory of modules of atypicality $i$ is denoted $\calA_i$. 

{\it Khovanov algebras}. We review some facts from the articles by Brundan and Stroppel  \cite{Brundan-Stroppel-1}, \cite{Brundan-Stroppel-2}, \cite{Brundan-Stroppel-3}, \cite{Brundan-Stroppel-4},  \cite{Brundan-Stroppel-5}. We denote the Khovanov-algebra of \cite{Brundan-Stroppel-4} associated to $GL(m|n)$ by $K(m|n)$. These algebras are naturally graded. For $K(m|n)$ we have a set of weights or weight diagrams which parametrise the irreducible modules (up to a grading shift). This set of weights is again denoted $X^+$. For each weight $\lambda \in X^+$ we have the irreducible module $L(\lambda)$, the indecomposable projective module $P(\lambda)$ with top $L(\lambda)$ and the standard or cell module $V(\lambda)$. If we forget the grading structure on the $K(m|n)$-modules, the main result of \cite{Brundan-Stroppel-4} is:

\begin{thm} There is an equivalence of categories $E$ from $\calR_{m|n}$ to the category of finite-dimensional left-$K(m|n)$-modules such that $EL(\lambda) = L(\lambda)$, $EP(\lambda) = P(\lambda)$ and $EK(\lambda) = V(\lambda)$ for $\lambda \in X^+$.
\end{thm}

$E$ is a Morita equivalence, hence $E$ will preserve the Loewy structure of indecomposable modules. This will enable us to study questions regarding extensions or Loewy structures in the category of Khovanov modules. 
%
%More precisely $K(m|n)$ is isomorphic to the locally finite endomorphism algebra $End_G^{fin}(P)^{op}$ of a canonical minimal projective generator $P \simeq \bigoplus_{\lambda \in X^+} P(\lambda)$ for $\calR_{m|n}$. In particular $E$ is a Morita equivalence. Hence $E$ will preserve the Loewy structure of indecomposable modules. This will enable us to study questions regarding extensions or Loewy structures in the category of Khovanov modules. 

{\it Weight diagrams}. To each highest weight $\lambda\in X^+$  we associate, following \cite{Brundan-Stroppel-4}, two subsets of cardinality $m$ respectively $n$ of the numberline $\Z$
\begin{align*} I_\times(\lambda)\ & =\ \{ \lambda_1  , \lambda_2 - 1, .... , \lambda_m - m +1 \} \\
 I_\circ(\lambda)\ & = \ \{ 1 - m - \lambda_{m+1}  , 2 - m - \lambda_{m+2} , .... ,  n-m- \lambda_{m+n}  \}. \end{align*}

The integers in $ I_\times(\lambda) \cap I_\circ(\lambda) $ are labeled by $\vee$, the remaining ones in $I_\times(\lambda)$ respectively $I_\circ(\lambda)$ are labeled by $\times$ respectively $\circ$. All other integers are labeled by a $\wedge$. 
This labeling of the numberline $\Z$ uniquely characterizes the weight $\lambda$. If the label $\vee$ occurs $r$ times in the labeling, then $r$ is called the degree of atypicality of $\lambda$. Notice that $0 \leq r \leq n$, and $\lambda$ is called maximal atypical if $r=n$. This notion of atypicality agrees with the previous one.

{\it Blocks}. Two irreducible representations $L(\lambda)$ and $L(\mu)$ in $\calR_{m|n}$ are in the same block if and only if the weights $\lambda$ and $\mu$ define labelings with the same position of the labels $\times$ and $\circ$. The degree of atypicality is a block invariant, and the blocks $\Lambda$ of atypicality $r$ are in 1-1 correspondence with pairs of disjoint subsets of $\bbZ$ of cardinality $m-r$ respectively $n-r$.

{\it Bruhat order}. The Bruhat order $\geq$ is the partial order on the set of weight diagrams generated by the operation of swapping a $\vee$ and a $\wedge$, so that getting bigger in the Bruhat order means moving $\vee$'s to the right.

{\it Cups and Caps}. To each such weight diagram with $r$ vertices labelled $\vee$ we associate its cup diagram as in \cite{Brundan-Stroppel-1}. Here a cup is a lower semi-circle joining two vertices. To construct the cup diagram one goes from the left to the right through the weight diagram until one finds a pair of vertices $\vee \ \ \ \wedge$ such that there are only $\times$'s, $\circ$'s or vertices which are already joined by cups between them. Then join $\vee \ \ \wedge$ by a cup. This procedure will result in a diagram with $r$ cups. Now remove all the labels of the vertices and draw rays down to infinity at all vertices labeled with a $\vee$ or $\wedge$ which are not part of a cup. If we draw the picture of a cup diagram we will not draw the rays. Analogously we define a cap to be an upper semi-circle joining two vertices. The cap diagram is build in the same way as the cup diagram. It is obtained from the latter by reflecting along the numberline. %As with the cup diagram we will not draw the rays in pictures.

\begin{example} As an example consider the trivial weight $(0,\ldots,0|0,\ldots,0)$ in $GL(n|n)$. Its weight diagram is given by 

\medskip
\begin{center}
%alpha
  \begin{tikzpicture}
 \draw (-5.5,0) -- (6.5,0);
\foreach \x in {-4,-3,-2,-1,0} %vee
     \draw (\x-.1, .2) -- (\x,0) -- (\x +.1, .2);
\foreach \x in {-5,1,2,3,4,5,6} %wedge
     \draw (\x-.1, -.2) -- (\x,0) -- (\x +.1, -.2);
\foreach \x in {} %cross
     \draw (\x-.1, .1) -- (\x +.1, -.1) (\x-.1, -.1) -- (\x +.1, .1);

\end{tikzpicture}
\end{center}
\medskip

with $n$ $\vee$'s at the vertices $-n+1,\ldots,0$. Its cup diagram is given by

\medskip
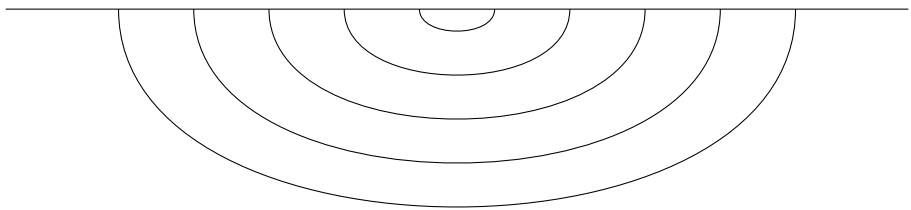
\begin{center}
%alpha
  \begin{tikzpicture}
 \draw (-5.5,0) -- (6.5,0);
\foreach \x in {} %vee
     \draw (\x-.1, .2) -- (\x,0) -- (\x +.1, .2);
\foreach \x in {} %wedge
     \draw (\x-.1, -.2) -- (\x,0) -- (\x +.1, -.2);
\foreach \x in {} %cross
     \draw (\x-.1, .1) -- (\x +.1, -.1) (\x-.1, -.1) -- (\x +.1, .1);

%cups
\draw [-,black,out=270,in=270](0,0) to (1,0);
\draw [-,black,out=270,in=270](-1,0) to (2,0);
\draw [-,black,out=270,in=270](-2,0) to (3,0);
\draw [-,black,out=270,in=270](-3,0) to (4,0);
\draw [-,black,out=270,in=270](-4,0) to (5,0);

\end{tikzpicture}
\end{center}
\medskip

\end{example}

%%%%%%%%%%%%%%%%%%%%%%%%%%%%%%%%%%%%%%%%%%%%%%%%%%%%%%%%%%%%%%%%%%%%%%%%%%%%%%%%

\section{Bipartitions and indecomposable modules}

For every $\delta \in k$ Deligne  \cite{Deligne-interpolation} \cite{Comes-Wilson} constructed a karoubian rigid symmetric monoidal category (called Deligne's interpolating category) which we denote by denoted $\uRep(GL_{\delta})$. This is a $k$-linear pseudo-abelian rigid tensor category. By construction it contains an object $st$ of dimension $\delta$, called the standard representation. 

\begin{thm}\cite{Deligne-interpolation}\label{Deligne-interpolation} Let $\mathcal{C}$ be a $k$-linear tensor category such that $End(\one) = k$. The functor $F \mapsto F(st)$ is an equivalence $Hom_{\mathcal{C}}^{\otimes} (\uRep(GL_{\delta}),\mathcal{C})$ of the tensor functors $\uRep(GL_{\delta}) \to \mathcal{C}$ with the category of objects in $\mathcal{C}$ which are dualisable of dimension $\delta$ and their isomorphisms.
\end{thm} 

%Given any $k$-linear pseudoabelian tensor category $C$ with unit object and a tensor functor $$F: Rep(GL_{\delta}) \to C,$$ the functor $F \to F(st)$ is an equivalence between the category of $\otimes$-functors of $Rep(GL_{\delta})$ to $C$ with the category of $\delta$-dimensional dualisable objects $X \in C$ and their isomorphisms. 

In particular, given a dualizable object $X$ of dimension $\delta$ in a $k$-linear pseudoabelian tensor category, a unique tensor functor $F_X: \uRep(GL_{\delta}) \to \mathcal{C}$ exists mapping $st$ to $X$.

Let $\lambda = (\lambda^L,\lambda^R)$ be a bipartition (a pair of partitions). Call $|\lambda| = |\lambda^L| + |\lambda^R|$ (where $|\lambda^L| = \sum \lambda_i^L$) the size (or degree) of the bipartition (notation $\lambda \vdash |\lambda|$) and $l(\lambda) = l(\lambda^L) + l(\lambda^R)$ the length of $\lambda$. We denote by $P$ the set of all partitions, by $\Lambda$ the set of all bipartitions. To each bipartition is attached an indecomposable object $R(\lambda)$ in $\uRep(GL_{\delta})$. By \cite{Comes-Wilson} the assignement $\lambda \to R(\lambda)$ defines a bijection between the set of bipartitions of arbitrary size and the set of isomorphism classes of nonzero indecomposable objects in $\uRep(GL_{\delta})$. By the universal property we have a tensor functor $F_{m|n}: \uRep(GL_d) \to \calR_{m|n}$ for $d = m-n$ given by standard representation of superdimension $m-n$. 

\begin{thm} \cite{Comes-Wilson} The image of $F_{m|n}$ is the space of mixed tensors, the full subcategory of objects which appear as a direct summand in a decomposition of \[ T(r,s): = V^{\otimes r} \otimes (V^{\vee})^{\otimes s} \] for some $r,s \in \N$. The functor $F_{m|n}$ is full. If $\lambda \neq \mu$, we have $F_{m|n}(R(\lambda)) \neq F_{m|n}(R(\mu))$.
\end{thm}

A bipartition is said to be $(m|n)$-cross if there exists some $1 \leq i \leq m+1$ with $\lambda_i^L + \lambda^R_{m+2-i} < n+1$. The set of $(m|n)$-cross bipartitions is denoted $\Lambda_{mn}^x$ or simply $\Lambda^x$. By abuse of notation we use the notation $R(\lambda)$ for $F_{m|n}(R(\lambda))$. By \cite{Comes-Wilson} the modules $R(\lambda)= F_{m|n}(R(\lambda))$ are $\neq 0$ if and only if $\lambda$ is an $(m|n)$-cross bipartition. Up to isomorphism the indecomposable nonzero summands of $V^{\otimes r} \otimes (V^{\vee})^{\otimes s}$ are the modules \cite{Brundan-Stroppel-5}, Thm 8.19, \[ \{ R(\lambda) \ | \ \lambda \in \dot\Lambda_{r,s} \ (m|n)-\text{cross } \} \] where ($\delta = m-n$) \begin{align*} \Lambda_{r,s} : = & \{ \lambda \in \Lambda^x \ | \ |\lambda^L| = r-t, \ |\lambda^R| = s-t \text{ for } \ 0 \leq t \leq min(r,s)\} \\ \dot\Lambda_{r,s} : =  & \begin{cases} \Lambda_{r,s}  \quad \quad \quad \quad & \text{ if } \delta\neq 0, \text{ or }  r\neq s \text{ or } r=s=0 \\ \Lambda_{r,s} \setminus (0,0) \ \ \ \ \ & \text{ if }   \delta = 0 \text{ and } r=s>0. \end{cases} \end{align*} 

For any bipartition $\lambda$ define the two sets   \begin{align*} I_{\wedge}(\lambda)  := \{ \lambda_1^L, \lambda_2^L - 1, \lambda_3^L - 2, \ldots \}, \ \ I_{\vee}(\lambda)  := \{1 - \delta -\lambda_1^R, 2 - \delta - \lambda_2^R, \ldots \}.\end{align*} Here we use the convention that a partition is always continued by an infinite number of zeros. To these two sets one can attach a weight diagram in the sense of \cite{Brundan-Stroppel-1} as follows: Label the integer vertices $i$ on the numberline by the symbols $\wedge, \vee, \circ, \times$ according to the rule \[ \begin{cases} \circ &\text{ if } \ i \  \notin I_{\wedge} \cup I_{\vee}, \\ \wedge \ &\text{ if } \ i \in I_{\wedge}, \ i \notin I_{\vee}, \\ \vee &\text{ if } \ i \in I_{\vee}, \ i \notin I_{\wedge}, \\ \times &\text{ if } \ i \in I_{\wedge} \cap I_{\vee}. \end{cases} \] 

We can attach to such a weight diagram a cup and cap diagram exactly as for the weight diagram of a highest weight.

\begin{remark} These weight diagrams differ from the weight diagrams attached to a highest weight. The weight diagram of a highest weight for $GL(m|n)$ has $r$ vertices labeled with a $\vee$ where $0 \leq r \leq n$ (the degree of atypicaliy), $m-r$ vertices labeled with a $\times$, $n-r$ vertices labeled with a $\circ$ and infinitely many $\wedge$'s. In the weight diagram of a bipartition with parameter $\delta$ the integer $i$ in is labelled $\wedge$ for $i << 0$. If $\delta \in \Z$, then $i$ is labelled by a $\vee$ for $i >> 0$ (hence for infinitely many vertices).
\end{remark}

\begin{example} The weight diagram of the bipartition $(\emptyset, \emptyset)$ for $\delta = 0$ is given by 

\medskip
\begin{center}
%alpha
  \begin{tikzpicture}
 \draw (-5.5,0) -- (6.5,0);
\foreach \x in {1,2,3,4,5} %vee
     \draw (\x-.1, .2) -- (\x,0) -- (\x +.1, .2);
\foreach \x in {-5,-4,-3,-2,-1,0} %wedge
     \draw (\x-.1, -.2) -- (\x,0) -- (\x +.1, -.2);
\foreach \x in {} %cross
     \draw (\x-.1, .1) -- (\x +.1, -.1) (\x-.1, -.1) -- (\x +.1, .1);

\end{tikzpicture}
\end{center}
\medskip

where the leftmost $\vee$ is at the vertex $1$. The associated cup diagram has no cups. The object $R(\emptyset,\emptyset)$ is the trivial object $\one$ in $\uRep(GL_{\delta})$ for any $\delta$.

\end{example}

\section{Definition of the modules $R(\lambda)$}\label{indecomposable}

The mixed tensors can be interpreted as the images of certain Khovanov-modules under the equivalence of categories $E^{-1}: K(m|n) \mbox{-}mod \to \calR_{m|n}$. This will give a way to identify the image $F_{m|n}(R(\lambda))$.

\textit{Some terminology of Brundan and Stroppel} \cite{Brundan-Stroppel-1} \cite{Brundan-Stroppel-2} \cite{Brundan-Stroppel-4}. Let $\alpha, \beta$ be weight diagrams for $K(m|n)$. Let $\alpha \sim \beta$ mean that $\beta$ can be obtained from $\alpha$ by permuting $\vee$'s and $\wedge$'s. The equivalence classes of this relation are called blocks. Given $\lambda, \mu \sim \alpha$ one can label the cup diagram $\lambda$ respectively the cap diagram $\mu$ with $\alpha$ to obtain $\underline{\lambda}\alpha$ resp. $\alpha \bar{\mu}$. These diagrams are by definition consistently oriented if and only if each cup respectively cap has exactly one $\vee$ and one $\wedge$ and all the rays labelled $\wedge$ are to the left of all rays labelled $\vee$. Put $\lambda \subset \alpha$ if and only if $\lambda \sim \alpha$ and $\underline{\lambda} \alpha$ is consistently oriented. 

\begin{example} Let $\alpha$ denote the weight diagram 

\medskip
\begin{center}
\begin{tikzpicture}
 \draw (-6,0) -- (6,0);
\foreach \x in {-3,0,1} %vee
     \draw (\x-.1, .2) -- (\x,0) -- (\x +.1, .2);
\foreach \x in {-5,-4,-2,-1,2,3,4} %wedge
     \draw (\x-.1, -.2) -- (\x,0) -- (\x +.1, -.2);
\foreach \x in {} %cross
     \draw (\x-.1, .1) -- (\x +.1, -.1) (\x-.1, -.1) -- (\x +.1, .1);

\end{tikzpicture}
\end{center}
\medskip

and consider the two weight diagrams $\lambda_1$ 

\medskip
\begin{center}
\begin{tikzpicture}
 \draw (-6,0) -- (6,0);
\foreach \x in {-3,-1,1} %vee
     \draw (\x-.1, .2) -- (\x,0) -- (\x +.1, .2);
\foreach \x in {-5,-4,-2,0,2,3,4} %wedge
     \draw (\x-.1, -.2) -- (\x,0) -- (\x +.1, -.2);
\foreach \x in {} %cross
     \draw (\x-.1, .1) -- (\x +.1, -.1) (\x-.1, -.1) -- (\x +.1, .1);

\end{tikzpicture}
\end{center}
\medskip

and $\lambda_2$

\medskip
\begin{center}
\begin{tikzpicture}
 \draw (-6,0) -- (6,0);
\foreach \x in {-5,-3,-1,0} %vee
     \draw (\x-.1, .2) -- (\x,0) -- (\x +.1, .2);
\foreach \x in {-4,-2,1,2,3,4} %wedge
     \draw (\x-.1, -.2) -- (\x,0) -- (\x +.1, -.2);
\foreach \x in {} %cross
     \draw (\x-.1, .1) -- (\x +.1, -.1) (\x-.1, -.1) -- (\x +.1, .1);

\end{tikzpicture}
\end{center}
\medskip

Then the labelled diagram $\underline{\lambda_1}\alpha$ is consistently oriented

\medskip
\begin{center}
\begin{tikzpicture}
 \draw (-6,0) -- (6,0);
\foreach \x in {-3,0,1} %vee
     \draw (\x-.1, .2) -- (\x,0) -- (\x +.1, .2);
\foreach \x in {-5,-4,-2,-1,2,3,4} %wedge
     \draw (\x-.1, -.2) -- (\x,0) -- (\x +.1, -.2);
\foreach \x in {} %cross
     \draw (\x-.1, .1) -- (\x +.1, -.1) (\x-.1, -.1) -- (\x +.1, .1);

%caps,cups
\draw [-,black,out=270, in=270](-3,0) to (-2,0);
\draw [-,black,out=270, in=270](-1,0) to (0,0);
\draw [-,black,out=270, in=270](1,0) to (2,0);

\end{tikzpicture}
\end{center}
\medskip

but the labelled diagram $\underline{\lambda_2}\alpha$ 

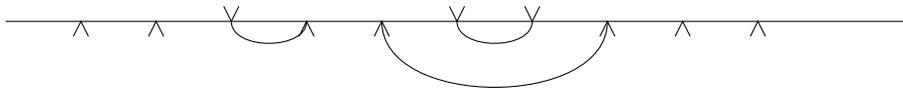
\medskip
\begin{center}
\begin{tikzpicture}
 \draw (-6,0) -- (6,0);
\foreach \x in {-3,0,1} %vee
     \draw (\x-.1, .2) -- (\x,0) -- (\x +.1, .2);
\foreach \x in {-5,-4,-2,-1,2,3,4} %wedge
     \draw (\x-.1, -.2) -- (\x,0) -- (\x +.1, -.2);
\foreach \x in {} %cross
     \draw (\x-.1, .1) -- (\x +.1, -.1) (\x-.1, -.1) -- (\x +.1, .1);

%caps,cups
\draw [-,black,out=270, in=270](-3,0) to (-2,0);
\draw [-,black,out=270, in=270](-1,0) to (2,0);
\draw [-,black,out=270, in=270](0,0) to (1,0);

\end{tikzpicture}
\end{center}
\medskip

is not consistently oriented. Hence $\lambda_1 \subset \alpha$, but $\lambda_2 \not\subset \alpha$.

\end{example}

A \textit{crossingless matching} is a diagram obtained by drawing a cap diagram underneath a cup diagram and then joining rays according to some order-preserving bijection between the vertices. Given blocks $\Delta, \Gamma$, a $\Delta \Gamma$-matching is a crossingless matching $t$ such that the free vertices (not part of cups, caps or lines) at the bottom are exactly at the same positions as the vertices labelled $\circ$ or $\times$ in $\Delta$; and similarly for the top with $\Gamma$. Given a $\Delta \Gamma$-matching $t$ and $\alpha \in \Delta$ and $\beta \in \Gamma$, one can label the bottom line with $\alpha$ and the upper line with $\beta$ to obtain $\alpha t \beta$. $\alpha t \beta$ is consistently oriented if each cup respectively cap has exactly one $\vee$ and one $\wedge$ and the endpoints of each line segment are labelled by the same symbol. Notation: $\alpha \rightarrow^t \beta$.

For a crossingless $\Delta \Gamma$-matching $t$ and $\lambda \in \Delta, \ \mu \in \Gamma$, label the bottom and the upper line as usual. The \textit{lower reduction} $red(\underline{\lambda}t)$ is the cup diagram obtained from $\underline{\lambda}t$ by removing the bottom number line and all connected components that do not extend up to the top number line. Circles and lines that do not cross the top number line are also called lower circles or lines. The upper reduction $red(t\bar{\mu})$ is the cap diagram obtained from $t\bar{\mu}$ by removing the top line.

\begin{example} Consider the following matching $t$ where we have labelled the upper number line with the weight diagram of the trival representation of $GL(3|3)$:

\medskip
%alpha
  
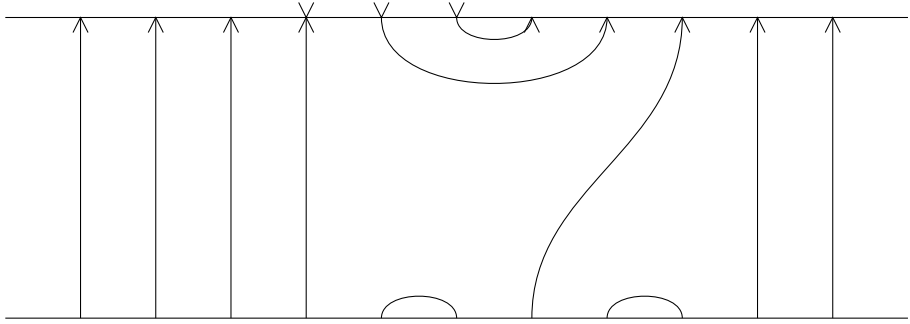
\begin{tikzpicture}
 \draw (-6,0) -- (6,0);
\foreach \x in {0,-1,-2} %vee
     \draw (\x-.1, .2) -- (\x,0) -- (\x +.1, .2);
\foreach \x in {-5,-4,-3,-2,1,2,3,4,5} %wedge
     \draw (\x-.1, -.2) -- (\x,0) -- (\x +.1, -.2);
\foreach \x in {} %cross
     \draw (\x-.1, .1) -- (\x +.1, -.1) (\x-.1, -.1) -- (\x +.1, .1);

\begin{scope} [yshift = -4 cm]

%lambda
 \draw (-6,0) -- (6,0);
\foreach \x in {} %vee
     \draw (\x-.1, .2) -- (\x,0) -- (\x +.1, .2);
\foreach \x in {} %wedge
     \draw (\x-.1, -.2) -- (\x,0) -- (\x +.1, -.2);
\foreach \x in {} %cross
     \draw (\x-.1, .1) -- (\x +.1, -.1) (\x-.1, -.1) -- (\x +.1, .1);
\end{scope}

%Verbindungen
\draw [-,black,out=270,in=90](-5,0) to (-5,-4);
\draw [-,black,out=270,in=90](-4,0) to (-4,-4);
\draw [-,black,out=270,in=90](-3,0) to (-3,-4);
\draw [-,black,out=270,in=90](-2,0) to (-2,-4);
\draw [-,black,out=270,in=90](3,0) to (1,-4);
\draw [-,black,out=270,in=90](4,0) to (4,-4);
\draw [-,black,out=270,in=90](5,0) to (5,-4);

%caps,cups
\draw [-,black,out=270, in=270](0,0) to (1,0);
\draw [-,black,out=270, in=270](-1,0) to (2,0);
\draw [-,black,out=90, in=90](-1,-4) to (0,-4);
\draw [-,black,out=90, in=90](2,-4) to (3,-4);

\end{tikzpicture}
\medskip

We want to label the lower numberline by the weight diagram $\alpha_i$ of a $3$-fold atypical weight of $GL(3|3)$ such that we get a consistently oriented matching. Then we have four choices for $\alpha_i$ (corresponding to the four choices for the distribution of $\vee$ and $\wedge$ in the lower cups):

\begin{enumerate}
 \item Choice 1: $\vee \wedge \vee \wedge$, $\alpha_1 = [2,0,0]$.
 \item Choice 2: $\vee \wedge \wedge \vee$, $\alpha_2 = [3,0,0]$.
 \item Choice 3: $\wedge \vee \vee \wedge$, $\alpha_3 = [2,1,0]$.
 \item Choice 4: $\wedge \vee \wedge \vee$  $\alpha_4 = [3,1,0]$.
\end{enumerate}

The weight $\lambda = [2, 0; 0]$ satisfies  $\lambda \subset \alpha_3$. We calculate the lower reduction $red(\underline{\lambda}t)$ by gluing the cup diagram of $[2,0,0]$ under the lower numberline

\medskip

%alpha
  
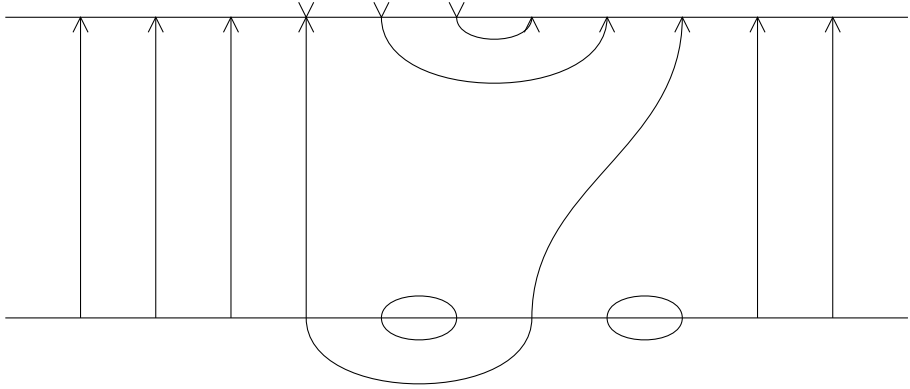
\begin{tikzpicture}
 \draw (-6,0) -- (6,0);
\foreach \x in {0,-1,-2} %vee
     \draw (\x-.1, .2) -- (\x,0) -- (\x +.1, .2);
\foreach \x in {-5,-4,-3,-2,1,2,3,4,5} %wedge
     \draw (\x-.1, -.2) -- (\x,0) -- (\x +.1, -.2);
\foreach \x in {} %cross
     \draw (\x-.1, .1) -- (\x +.1, -.1) (\x-.1, -.1) -- (\x +.1, .1);

\begin{scope} [yshift = -4 cm]

%lambda
 \draw (-6,0) -- (6,0);
\foreach \x in {} %vee
     \draw (\x-.1, .2) -- (\x,0) -- (\x +.1, .2);
\foreach \x in {} %wedge
     \draw (\x-.1, -.2) -- (\x,0) -- (\x +.1, -.2);
\foreach \x in {} %cross
     \draw (\x-.1, .1) -- (\x +.1, -.1) (\x-.1, -.1) -- (\x +.1, .1);
\end{scope}

%Verbindungen
\draw [-,black,out=270,in=90](-5,0) to (-5,-4);
\draw [-,black,out=270,in=90](-4,0) to (-4,-4);
\draw [-,black,out=270,in=90](-3,0) to (-3,-4);
\draw [-,black,out=270,in=90](-2,0) to (-2,-4);
\draw [-,black,out=270,in=90](3,0) to (1,-4);
\draw [-,black,out=270,in=90](4,0) to (4,-4);
\draw [-,black,out=270,in=90](5,0) to (5,-4);

%caps,cups
\draw [-,black,out=270, in=270](0,0) to (1,0);
\draw [-,black,out=270, in=270](-1,0) to (2,0);
\draw [-,black,out=90, in=90](-1,-4) to (0,-4);
\draw [-,black,out=90, in=90](2,-4) to (3,-4);

%caps,cups
\draw [-,black,out=270, in=270](-1,-4) to (0,-4);
\draw [-,black,out=270, in=270](2,-4) to (3,-4);
\draw [-,black,out=270, in=270](-2,-4) to (1,-4);

\end{tikzpicture}

\medskip

We see that the lower reduction is the cup diagram of the trivial representation. What we have checked is that $\lambda$ satisfies $\lambda \subset \alpha_3 \rightarrow^t \, \ red(\underline{\lambda}t) = \underline{\zeta}$ where $\zeta$ denotes the weight diagram of the trivial representation.
\end{example}

If $M = \bigoplus_{j \in \Z} M_j$ is a graded $K(m|n)$-module, write $M\langle j \rangle$ for the same module with new grading $M \langle j \rangle_i := M_{i-j}$. The modules $\{ L(\lambda)\langle j \rangle \ | \ \lambda \in X^+, \ j \in \Z \}$ give a complete set of representatives for the isomorphism classes of irreducible graded $K(m|n)$-modules. The Grothendieck group is the free $\Z$-module with basis the $L(\lambda)\langle j \rangle$. Viewing it instead as a $\Z[q,q^{-1}]$-module so that $q^j [M] : = [M\langle j \rangle]$, $K_0(Rep(K(m|n))$ becomes the free $\Z[q,q^{-1}]$-module with basis $\{ L(\lambda) \ | \ \lambda \in X^+ \}$.

For any $\Delta \Gamma$-matching $t$ we have the special projective functors $G_{\Delta \Gamma}^t$ in the category of graded $K(m|n)$-modules \cite{Brundan-Stroppel-2}.
The mixed tensors $R(\lambda)$ will be the images of certain special cases of these projective functors under $E$. Given a bipartition $\lambda$ we denote by the defect $d(\lambda)$ of $\lambda$ the number of caps in the cap diagram and by the rank of $\lambda$ $rk(\lambda) = min (\# \times,\#\circ)$ the minimal number of $\times$' or $\circ$'s, whichever is smaller. For $\delta \geq 0$ one has $rk(\lambda) = \# \circ$'s. Then put \[ k(\lambda):= d(\lambda) + rk(\lambda).\]

We now define the modules $R(\lambda)$ attached to an $(m|n)$-cross bipartition $\lambda$. For this we first associate to any $\lambda \in \Lambda^x$ a highest weight $\lambda^{\dagger}$. This highest weight will be the heighest weight of the irreducucible socle of $R(\lambda)$. We now define a map $\theta: \Lambda^x \to X^+$ following \cite{Brundan-Stroppel-5}. Fix a bipartition $\lambda \in \Lambda^x$ and denote by $\eta$ the weight diagram \cite[6.1]{Brundan-Stroppel-5} of the trivial bipartition $(\emptyset,\emptyset)$

\medskip
\begin{center}
\begin{tikzpicture}
 \draw (-6,0) -- (6,0);
\foreach \x in {2,3,4,5} %vee
     \draw (\x-.1, .2) -- (\x,0) -- (\x +.1, .2);
\foreach \x in {-5,-4,-3} %wedge
     \draw (\x-.1, -.2) -- (\x,0) -- (\x +.1, -.2);
\foreach \x in {-2,-1,0,1} %cross
     \draw (\x-.1, .1) -- (\x +.1, -.1) (\x-.1, -.1) -- (\x +.1, .1);

\end{tikzpicture}
\end{center}
\medskip

where the rightmost $\times$ is at position zero, and there are $\delta = m-n$ crosses. Let $\alpha$ be the weight diagram obtained from $\eta$ by switching the rightmost $k(\lambda)$ $\wedge$'s with the leftmost $k(\lambda)$ $\vee$'s. Let $t$ be the crossingless matching between $\bar{\lambda}$ and $\underline{\alpha}$ obtained as follows: Draw the cap diagram $\overline{\lambda}$ underneath the cup diagram $\underline{\alpha}$ and then join the rays in the unique way such that rays coming from a vertex $a \in \Z$ get joined with rays coming from the vertex $a$ except for a finite number of vertices. Now replace $\alpha$ with the weight diagram of the trivial representation denoted $\zeta$. Then adjust the labels of $\lambda$ that are at the bottoms of line segments to obtain $\lambda^{\dagger}$ such that $\lambda^{\dagger}t\zeta$ is consistently oriented.

Let $\Gamma$ be the block of $\zeta$, $\Delta$ be the block containing $\lambda^{\dagger}$ and put  \[ R(\lambda) = G_{\Delta \Gamma}^t L(\zeta) \] where $G_{\Delta \Gamma}^t$ is a special projective functor from \cite{Brundan-Stroppel-2} \cite{Brundan-Stroppel-5}. We transport $R(\lambda)$ by the equivalence of categories $E: \calR_{m|n} \to K(m|n)\mbox{-}mod$. Then $R(\lambda)$ agrees with $F_{m|n}(R(\lambda))$ by \cite{Brundan-Stroppel-5}. By Morita equivalence the Loewy layers are preserved. We denote by $\lambda^{\dagger}$ the highest weight of the irreducible socle of $R(\lambda)$. This defines a map $\theta: \Lambda^x \to X^+$, $\lambda \mapsto \lambda^{\dagger}$.

\begin{example} We calculate $\lambda^{\dagger}$ for $\lambda = (1^4;1)$ in $\calR_{4|1}$. The weight diagram of $\lambda$ is 

\medskip

\begin{center}
\begin{tikzpicture}
 \draw (-6,0) -- (6,0);
\foreach \x in {-3,2,3,4,5} %vee
     \draw (\x-.1, .2) -- (\x,0) -- (\x +.1, .2);
\foreach \x in {-5,-4,-2} %wedge
     \draw (\x-.1, -.2) -- (\x,0) -- (\x +.1, -.2);
\foreach \x in {-1,0,1} %cross
     \draw (\x-.1, .1) -- (\x +.1, -.1) (\x-.1, -.1) -- (\x +.1, .1);

\end{tikzpicture}
\end{center}
\medskip

with the leftmost $\vee$ at the vertex $-3$; and the weight diagram of $\alpha$ is 

\medskip

\begin{center}
\begin{tikzpicture}
 \draw (-6,0) -- (6,0);
\foreach \x in {-3,2,3,4,5} %vee
     \draw (\x-.1, .2) -- (\x,0) -- (\x +.1, .2);
\foreach \x in {-5,-4,1} %wedge
     \draw (\x-.1, -.2) -- (\x,0) -- (\x +.1, -.2);
\foreach \x in {-2,-1,0} %cross
     \draw (\x-.1, .1) -- (\x +.1, -.1) (\x-.1, -.1) -- (\x +.1, .1);

\end{tikzpicture}
\end{center}
\medskip

with the leftmost $\vee$ at the vertex $-3$. Hence the crossingless matching $t$ between $\underline{\alpha}$ and $\bar{\lambda}$ is 

\medskip

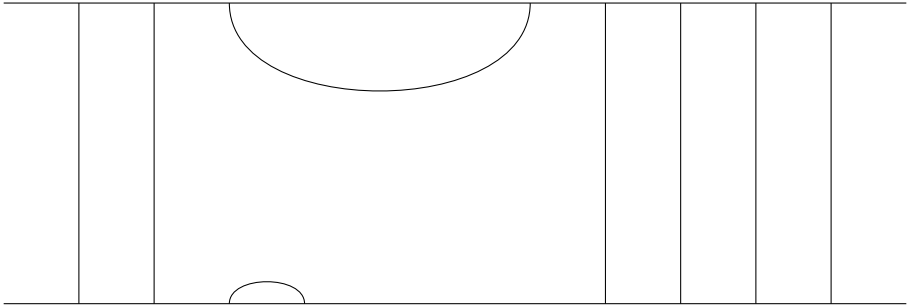
\begin{center}
%alpha
  \begin{tikzpicture}
 \draw (-6,0) -- (6,0);
\foreach \x in {} %vee
     \draw (\x-.1, .2) -- (\x,0) -- (\x +.1, .2);
\foreach \x in {} %wedge
     \draw (\x-.1, -.2) -- (\x,0) -- (\x +.1, -.2);
\foreach \x in {} %cross
     \draw (\x-.1, .1) -- (\x +.1, -.1) (\x-.1, -.1) -- (\x +.1, .1);

\begin{scope} [yshift = -4 cm]

%lambda
 \draw (-6,0) -- (6,0);
\foreach \x in {} %vee
     \draw (\x-.1, .2) -- (\x,0) -- (\x +.1, .2);
\foreach \x in {} %wedge
     \draw (\x-.1, -.2) -- (\x,0) -- (\x +.1, -.2);
\foreach \x in {} %cross
     \draw (\x-.1, .1) -- (\x +.1, -.1) (\x-.1, -.1) -- (\x +.1, .1);
\end{scope}

%Verbindungen
\draw [-,black,out=270,in=90](-5,0) to (-5,-4);
\draw [-,black,out=270,in=90](-4,0) to (-4,-4);
\draw [-,black,out=270,in=90](2,0) to (2,-4);
\draw [-,black,out=270,in=90](3,0) to (3,-4);
\draw [-,black,out=270,in=90](4,0) to (4,-4);
\draw [-,black,out=270,in=90](5,0) to (5,-4);

%caps,cups
\draw [-,black,out=270, in=270](-3,0) to (1,0);
\draw [-,black,out=90, in=90](-3,-4) to (-2,-4);

\end{tikzpicture}

\end{center}
\medskip

If we label the upper numberline by $\zeta$ and the lower by one by $\lambda$ (so we look at $\lambda t \zeta$) we obtain

\medskip
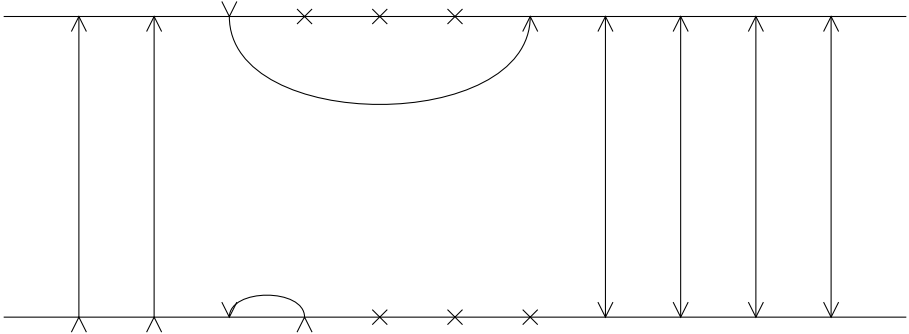
\begin{center}

%alpha
  \begin{tikzpicture}
 \draw (-6,0) -- (6,0);
\foreach \x in {-3} %vee
     \draw (\x-.1, .2) -- (\x,0) -- (\x +.1, .2);
\foreach \x in {-5,-4,1,2,3,4,5} %wedge
     \draw (\x-.1, -.2) -- (\x,0) -- (\x +.1, -.2);
\foreach \x in {-2,-1,0} %cross
     \draw (\x-.1, .1) -- (\x +.1, -.1) (\x-.1, -.1) -- (\x +.1, .1);

\begin{scope} [yshift = -4 cm]

%lambda
 \draw (-6,0) -- (6,0);
\foreach \x in {-3,2,3,4,5} %vee
     \draw (\x-.1, .2) -- (\x,0) -- (\x +.1, .2);
\foreach \x in {-5,-4,-2} %wedge
     \draw (\x-.1, -.2) -- (\x,0) -- (\x +.1, -.2);
\foreach \x in {-1,0,1} %cross
     \draw (\x-.1, .1) -- (\x +.1, -.1) (\x-.1, -.1) -- (\x +.1, .1);
\end{scope}

%Verbindungen
\draw [-,black,out=270,in=90](-5,0) to (-5,-4);
\draw [-,black,out=270,in=90](-4,0) to (-4,-4);
\draw [-,black,out=270,in=90](2,0) to (2,-4);
\draw [-,black,out=270,in=90](3,0) to (3,-4);
\draw [-,black,out=270,in=90](4,0) to (4,-4);
\draw [-,black,out=270,in=90](5,0) to (5,-4);

%caps,cups
\draw [-,black,out=270, in=270](-3,0) to (1,0);
\draw [-,black,out=90, in=90](-3,-4) to (-2,-4);

\end{tikzpicture}

\end{center}
\medskip

The matching $\lambda t \zeta$ is not consistently orientied any more since at all vertices $\geq 2$ the rays have different labels ($\wedge$'s in $\zeta$, $\vee$'s in $\lambda$). Adjusting the labels in $\lambda$ to make it consistent means turning all $\vee$'s at vertices $\geq 2$ to $\wedge$'s. The diagram obtained is the weight diagram of the highest weight $\lambda^{\dagger}$ 

\medskip

\begin{center}
\begin{tikzpicture}
 \draw (-6,0) -- (6,0);
\foreach \x in {-3} %vee
     \draw (\x-.1, .2) -- (\x,0) -- (\x +.1, .2);
\foreach \x in {-5,-4,-2,2,3,4,5} %wedge
     \draw (\x-.1, -.2) -- (\x,0) -- (\x +.1, -.2);
\foreach \x in {-1,0,1} %cross
     \draw (\x-.1, .1) -- (\x +.1, -.1) (\x-.1, -.1) -- (\x +.1, .1);

\end{tikzpicture}
\end{center}
\medskip

with the $\vee$ at the vertex $-3$; and therefore the irreducible socle is given by $L(1,1,1,0|0)$.
\end{example}

%-----------------------------------------------------

%-----------------------------------------------------

%-----------------------------------------------------

\section{Irreducible modules and projective covers}

We describe the $R(\lambda)$ which are irreducible and those which are projective.

\begin{thm} \cite[Theorem 3.4]{Brundan-Stroppel-5} \cite[Theorem 4.11]{Brundan-Stroppel-2} \label{thm:graded-multiplicities} (i) For a given $\Delta \Gamma$-matching $t$, $G_{\Delta \Gamma}^t L(\mu)$ is an indecomposable module with irreducible head and socle which differ only by a grading shift.
(ii) In the graded Grothendieck group \[ [ G_{\Delta \Gamma}^t L(\mu) ] = \sum_{\gamma} (q + q^{-1})^{n_{\gamma}} [L(\gamma)] \] where $n_{\gamma}$ denotes the number of lower circles in $\underline{\gamma}t$ and the sum is over all $\gamma \in \Delta$ such that a) $\underline{\mu}$ is the lower reduction of $\underline{\gamma}t$ and b) the rays of each lower line in $\underline{\gamma}\gamma t$ are oriented so that exactly one is $\vee$ and one is $\wedge$.
 (iii) If we forget the grading then \[ [ G_{\Delta \Gamma}^t L(\mu)] = \sum_{ \gamma \subset \beta \rightarrow^t \mu, \ red(\underline{\lambda}t) = \underline{\mu} } [ L(\gamma) ].\]
\end{thm}

\begin{remark} Note that (iii) does not mean that the multiplicity of a composition factor is always 1. In the summation over $\gamma \subset \beta \rightarrow^t \mu$ the matching $t$ and the labeling of the upper number line in $t$ by $\mu$ are fixed. In a first step we determine all possible labelings $\beta$ of the lower numberline such that the labeled matching is consistently oriented ($\beta \rightarrow^t \mu$). In a second step we determine for each $\beta$ all possible $\gamma$ with $\gamma \subset \beta$. In this way we can get composition factors $L(\gamma)$ with multiplicity $>1$. In a last step we have to discard those which do not satisfy $red(\underline{\gamma}t) = \underline{\mu}$. For an example see \ref{comp-factor-example} below. It follows from Theorem \ref{thm:graded-multiplicities} (ii) that the multiplicity of a composition factor $L(\gamma)$ is $2^{n_{\gamma}}$.
\end{remark}

\begin{example} \label{comp-factor-example} We calculate the composition factors of $R(1^4;1)$ in $Gl(4|1)$. Note that the composition factors depend on $\delta$. Since $\beta \to^t \zeta$ has to be consistently oriented, this fixes $\beta$ up to the labeling of the unique cup, hence such a $\beta$ is either $\lambda^{\dagger}$ or $\tilde{\beta}$

\medskip

\begin{center}
\begin{tikzpicture}
 \draw (-6,0) -- (6,0);
\foreach \x in {-2} %vee
     \draw (\x-.1, .2) -- (\x,0) -- (\x +.1, .2);
\foreach \x in {-5,-4,-3,2,3,4,5} %wedge
     \draw (\x-.1, -.2) -- (\x,0) -- (\x +.1, -.2);
\foreach \x in {-1,0,1} %cross
     \draw (\x-.1, .1) -- (\x +.1, -.1) (\x-.1, -.1) -- (\x +.1, .1);

\end{tikzpicture}
\end{center}
\medskip

Since there is only one cup in $\underline{\lambda^{\dagger}}$, the set of $\gamma$'s such that $\gamma \subset \lambda^{\dagger}$ consists of $\lambda^{\dagger}$ and 

\medskip

\begin{center}
\begin{tikzpicture}
 \draw (-6,0) -- (6,0);
\foreach \x in {-4} %vee
     \draw (\x-.1, .2) -- (\x,0) -- (\x +.1, .2);
\foreach \x in {-5,-3,-2,2,3,4,5} %wedge
     \draw (\x-.1, -.2) -- (\x,0) -- (\x +.1, -.2);
\foreach \x in {-1,0,1} %cross
     \draw (\x-.1, .1) -- (\x +.1, -.1) (\x-.1, -.1) -- (\x +.1, .1);

\end{tikzpicture}
\end{center}
\medskip

with the $\vee$ at the vertex -4. Now for $\tilde{\beta}$ the set of $\gamma$'s such that $\gamma \subset \tilde{\beta}$ consists of $\tilde{\beta}$ and $\lambda^{\dagger}$. The condition $red(\underline{\gamma}t) = \underline{\zeta}$ is satisfied for all composition factors. It can now be easily seen that the socle filtration is \[ \begin{pmatrix} L(1,1,1,0|0) \\ L(1,1,1,1|-1) + L(1,1,1,-1|1) \\ L(1,1,1,0|0) \end{pmatrix}. \] The module $R(1^4;1)$ is the projective cover of the singly atypical module $L(1,1,1,0|0)$.
\end{example}

%The information about the \textit{graded composition multiplicities} is finer than the mere information about the composition factors since it gives rise to a grading filtration with semisimple quotients. 

\begin{cor}\label{loewy-length} $R(\lambda)$ has Loewy length $2 d(\lambda) + 1$. It is rigid.
\end{cor}

\begin{proof} Let $R(j)$ be the submodule of $R(\lambda)$ spanned by all graded pieces of degree $\geq j$. Then \[ R(\lambda) = R(- d(\lambda)) \supset R(- d(\lambda) +1) \supset \ldots \supset  R(d(\lambda))\] with successive semisimple quotients $R(j)/R(j+1)$ of degree $j$. By \cite{Brundan-Stroppel-4} every block of ${\calR}_{mn}$ is Koszul. We already know that the top and socle are simple. Since Koszul algebras are quadratic, the following proposition finishes the proof. 
\end{proof}

\begin{prop} \cite[Proposition 2.4.1]{Beilinson-Ginzburg-Soergel} Let $A$ be a graded ring such that i) $A_0$ is semisimple, ii) A is generated by $A_1$ over $A_0$. Let $M$ be a graded $A$-module of finite length. If $soc(M)$ (resp. $top(M))$ is simple, the socle (resp. the radical) filtration on $M$ coincides with the grading filtration (up to a  shift).
\end{prop}

\begin{cor} Every indecomposable module in $\calR$ with irreducible top and socle is rigid.
\end{cor}

\begin{cor} $R(\lambda)$ is irreducible if and only if $d(\lambda) = 0$.
\end{cor}

\subsection{Tensor generators} A representation $X$ of a supergroup $G$ is a tensor generator if every representation is a quotient of a finite direct sum of representations $X^{\otimes r} \otimes (X^{\vee})^{\otimes s}$ for some $r,s \geq 0$. If $G$ is an algebraic group, every faithful representation is a tensor generator. The same statement is true for representations of supergroups as observed first by Weissauer.
%
 %In the $GL(m|n)$-case it is easily seen that $V$ is  a tensor generator of $\calR_{m|n}$ either by adapting the classical proof \cite{Deligne-hodge} \cite{Milne-affine} or by reducing the proof to the classical case. This can be done using the splitting theorem of Weissauer \cite{Weissauer-semisimple} or Masuoka \cite{Masuoka-hopf} stating that $k[GL(m|n)] = k[GL(m) \times GL(n)] \otimes (\mathfrak{g}_-)^*$ where $GL(m) \times GL(n)$ is the underlying classical group and $(\mathfrak{g}_-)^*$ is the $k$-dual of the odd part of the underlying Lie superalgebra associated to $G$. Note that we have the same equivalence of categories between $k[G]$-comodules and representations of $G$ as in the classical case \cite{Weissauer-semisimple} \cite{Masuoka-hopf}. More generally consider immersive representations $\rho:G \to GL(V), \ V \simeq k^{m|n}$, i.e. $\rho$ is injective on the level of the underlying classical groups and on the Lie superalgebra level. The following theorem can be easily proven using the splitting theorem.

\begin{thm} \cite[Section 7.1]{Comes-Heidersdorf} Let $\rho: G \to GL(V)$ be a faithful representation. Then any finite dimensional $k[G]$-comodule is a quotient of a finite multiple of some iterated tensor product of the $k[G]$-comodules $V$ and $V^{\vee}$.
\end{thm}

\subsection{Projective covers.} Recall that the indecomposable projective modules in $Rep(GL(m|n))$ are the irreducible typical modules and the projective covers of the irreducible atypical modules by \cite{Kac-Rep} .

\begin{lem}\label{projective-mixed-tensors-1} Every indecomposable projective module appears as some $R(\lambda)$.
\end{lem}

\begin{proof} The module $V$ is a tensor generator of $\calR_{m|n}$. Hence every module $M \in \calR$ appears as a subquotient of some direct sum of $T(r,s)$. If $M$ is indecomposable projective the surjection will split, hence $M$ appears as a direct summand. 
\end{proof}

Since every atypical weight appears in the socle and top of its projective cover we obtain also

\begin{cor} The map $\theta: \Lambda^x \to X^+$ is surjective.
\end{cor}

%Now that we know that every projective cover appears as some $R(\lambda)$, we characterize the projective covers in this part. 

\begin{lem} The crosses and circles of the bipartition $\lambda$ are at the same vertices as the crosses and circles of the highest weight $\lambda^{\dagger}$. In particular $at(R(\lambda)) = n- rk(\lambda)$.
\end{lem}

\begin{proof} Only the labels $\vee$ or $\wedge$ in the weight diagram of $\lambda$ are changed when applying $\theta$.
\end{proof}

We use the following notation: If $\lambda$ is a weight or weight diagram, we write $\lambda(i)$ for the $i$-th vertex.

\begin{thm}\label{projective-mixed-tensors-2} A mixed tensor $R(\lambda)$ is projective if and only if $k(\lambda) = n$. In this case $R(\lambda) = P(\lambda^{\dagger})$. 
\end{thm}

\begin{proof} For every indecomposable  module $M$ with $top(M) = L(\lambda^{\dagger})$ there exists a surjection $P(\lambda^{\dagger}) \to M$ by \cite{Zou}, lemma 3.4. If $M$ has the same composition factors as $P(\lambda)$, this surjection has to have trivial kernel and gives an isomorphism. By \cite{Brundan-Stroppel-4} the following formulas hold in the Grothendieck group \[ [P(\lambda^{\dagger})] = \sum_{\mu \supset \lambda^{\dagger}} [K(\mu)], \ \ \ \ [K(\mu)] = \sum_{\gamma \subset \mu} [L(\gamma)].\] On the other hand \[ [R(\lambda)] = [G_{\Delta \Gamma}^t L(\zeta)] = \sum_{\gamma \subset \beta \to^t \zeta, red(\underline{\gamma}t) = \underline{\zeta}} \ [L(\gamma)].\] We will show that we get the same composition factors for $P(\lambda^{\dagger})$ and $R(\lambda)$ if $k(\lambda) = n$. Assume now that $k(\lambda) = n$. Since $\zeta$ and $t$ are fixed, the conditions $\beta \to^t \zeta$ implies that $\beta(i)$ is fixed up to the choice of the position of $\vee$ and $\wedge$ in each cup: All other coordinates are determined by the condition that the endpoints of line segments of $t$ must be labelled by the same symbol (and implies that $\beta$ has $n$ cups and no free $\vee$'s). Hence any such $\beta$ differs from $\lambda^{\dagger}$ only by the position of $\vee$ and $\wedge$ in each cup. The set of $\beta$ so obtained is the set of $\beta$ with $\beta \supset \lambda^{\dagger}$: the condition that there cannot be free $\vee$'s to the left of free $\wedge$'s forces all $n$ $\vee$'s to be bound in cups. Hence \[ R(\lambda) = G_{\Delta \Gamma}^t L(\zeta) = \sum_{\gamma \subset \beta \supset \lambda^{\dagger}, red(\underline{\gamma}t) = \underline{\zeta}} \ [L(\gamma)].\] It is easy to see that the condition $red(\underline{\gamma}t) = \underline{\zeta}$ is always satisfied for $k(\lambda) = n$, hence we know $R(\lambda) = P(\lambda^{\dagger})$ for $k(\lambda) =n$. If $k(\lambda) = n$, the condition $\beta \to^t \zeta$ is equivalent to $\beta \supset \lambda^{\dagger}$. For $k(\lambda) = n-r$, $r >0$, the condition $\beta \to^t \zeta$ is stricter than the condition $\beta \supset \lambda^{\dagger}$. Hence the composition factors of $R(\lambda)$ are just a proper subset of the ones of $P(\lambda^{\dagger})$. 
\end{proof}

%\textbf{Example.} The module $R((3,2,1),(3,2,1))$ is the projective cover $P([2,1,0])$ in $\calR_{3}$.

%\bigskip

%----------------------------------------------------

%----------------------------------------------------

\section{The socle of $R(\lambda)$}\label{sec:theta}

As noted by \cite{Brundan-Stroppel-5} the map $\theta: \Lambda^x \to X^+$ is in general not injective. %It is not even injective if one fixes the defect and the rank of the bipartition.

\begin{lem} $\theta$ is injective a) on the set of bipartitions satisfying $d(\lambda) =0$ and b) on the set of bipartitions satisfying $k(\lambda) = n$.
\end{lem}

\begin{proof}
If $k(\lambda) = n$, then $R(\lambda) = P(\lambda^{\dagger})$. If $d(\lambda) = 0$, then $R(\lambda) = L(\lambda^{\dagger})$. Both $P(\lambda^{\dagger})$ and $L(\lambda^{\dagger})$ are determined by their socle. We are done since $R(\lambda) = R(\mu)$ if and only if $\lambda = \mu$.
\end{proof}

Since $\theta$ is injective for minimal and maximal defect we can describe its inverse $\theta^{-1}: X^+ \to \Lambda$ in these two cases. Here and in the following we use implicitely the following obvious lemma. For the notation $\lambda t \alpha$ see the beginning of section \ref{indecomposable}. In the next lemma $\alpha$ refers to the weight diagram used in the definition of $R(\lambda)$.

\begin{lem} For $\lambda \in \Lambda^x$ the labelled matching $\lambda t \alpha$ is consistently oriented.
\end{lem}

As a consequence any vertex where $\alpha(i) \neq \zeta(i)$ will result in a switch of a label in $\lambda$ when passing from $\lambda \to \lambda^{\dagger}$. This results in the following simplified description for $\lambda \mapsto \lambda^{\dagger}$.

\subsection{An algorithm.} We now give a more explicit description of the map $\theta$, i.e. we describe how to transform the weight diagram of the bipartition $\lambda$ into the weight diagram of the highest weight $\lambda^{\dagger}$. The weight diagram $\alpha$ differs from $\zeta$ in the following way: To the left of the $m-n$ crosses we have $n - k(\lambda)$ different labels and to the right of the $m-n$ crosses infinitely many. Define $M$ to be the largest vertex labelled with a $\times$ or $\circ$ or part of a cup in the weight diagram of $\lambda$. The matching $t$ will be the identity (meaning $t$ connects the $i$-th vertex of $\alpha$ with the $i$-th vertex of $\lambda$) from vertices greater or equal to \[ T= max(k(\lambda) + 1, M+1).\] Since $\alpha(i) \neq \zeta(i)$ for all $i \geq T$, all labels in $\lambda$ at vertices greater or to $T$ will be switched (i.e. a $\vee$ turns into a $\wedge$ and vice versa) when passing from $\lambda$ to $\lambda^{\dagger}$. Now define \begin{align*} X= \begin{cases} 0 & M+1 \leq k(\lambda) + 1 \\ M- k(\lambda) & \text{else.} \end{cases} \end{align*} A free vertex is one which does not have a cross, or a circle or is not part of a cup. 

\begin{cor} The weight diagram of $\lambda^{\dagger}$ is obtained from the weight diagram of $\lambda$ by switching all labels at vertices $\geq T$ and switching the first $X + n-k(\lambda)$ free vertices $< T$. 
\end{cor}

\textbf{Example.} Let us consider typical weights. Say the $\times$ are at position $v_1 > v_2 > \ldots > v_m$ and the circles at position $w_1 > \ldots > w_n$. Then \begin{align*} \lambda_1^{\dagger} = v_1, & \lambda_2^{\dagger} = v_2 + 1, \ldots, \lambda_m^{\dagger} = v_m + m-1, \lambda_{m+1}^{\dagger} = w_n + m - 1, \ldots, \\ & \lambda_{m+n}^{\dagger} = w_1 + m -n.\end{align*} The inverse $\lambda^{\dagger} \mapsto \lambda$: Given any typical weight $\lambda^{\dagger}$ we distinguish two cases. Either $T = M+1$ where $M$ is the rightmost vertex labelled with $\times$ or $\circ$, or $T = n+1$. If $T=n+1$, all the free vertices up to $n$ are labelled with $\wedge$'s and the remaining ones to the right with $\vee$'s. Otherwise there will be $\vee$'s in the $T - n - 1$ free positions to the left of the rightmost cross or circle. If $T=M+1$ we switch all the labels (from $\wedge$ to $\vee$ and vice versa) at vertices $\geq M+1$ as well as the labels at the first $M-n$ free vertices left of $M$. This describes $\theta$ and $\theta^{-1}$ for $\lambda^{\dagger}$ typical.

%---------------------------------------

\subsection{The map $\theta$ in the typical case}  If $\lambda^{\dagger}$ is typical, an explicit expression for the two maps $\theta$ and $\theta^{-1}$ can be given in terms of the coordinates of the bipartition using \cite{Moens-van-der-Jeugt-character}, \cite{Moens-van-der-Jeugt-composite}. The authors define a subset $\Lambda^{st} \subset \Lambda^x$ and attach to such a bipartition $(\mu,\nu)$ (called $\mathfrak{gl}(m|n)$-standard) the highest weight $\tilde{\Theta}(\mu,\nu) =\lambda_{\mu,\nu}$. Conversely to any typical weight $\lambda^{\dagger}$ there is an attached bipartition $(\mu,\nu)$ \cite[lemma 3.15]{Moens}.

\begin{lem} Let $\lambda$ be such that $R(\lambda) = L(\lambda^{\dagger})$ is typical. Then $\lambda^{\dagger} = \lambda_{\lambda^L,\lambda^R}$ and the inverse $\theta^{-1}(\lambda^{\dagger})$ is given by the rule of  \cite[lemma 3.15]{Moens}
\end{lem}

\begin{proof} The set $\Lambda^{st}$ is a subset of $\Lambda^x$. Hence both $\lambda^{\dagger}$ and $\lambda_{\mu,\nu}$ are defined on $\Lambda^{st}$. Every typical weight in $X^+$ is in the image of $\tilde{\theta}$ by \cite[lemma 3.15]{Moens}. The character of $L(\lambda_{\mu,\nu})$ is computed in \cite{Moens-van-der-Jeugt-composite} and is given by the supersymmetric Schur function $s_{\mu,\nu}$. Similarly the character of $R(\lambda) = L(\lambda^{\dagger})$ is computed in \cite{Comes-Wilson}. The two characters are equal. Since the character determines the irreducible representation the result follows. 
\end{proof}

Note that the condition $\mathfrak{gl}(m|n)$-standard of loc.cit is not equivalent to the condition $(m|n)$-cross. Furthermore the map which associates to any bipartition the weight $\lambda_{\mu,\nu}$ does in general not agree with $\lambda \mapsto \lambda^{\dagger}$.

%-------------------------------------------

\subsection{Kostant weights} A weight $\mu$ is called a Kostant weight if the cup diagram of $L(\mu)$ is
completely nested. In other words if its weight diagram is $\wedge \vee \wedge
\vee$-avoiding in the sense that there are no vertices $i<j<k<l$ labelled in this order by $\wedge \vee \wedge \vee$.

\begin{lem} Every irreducible mixed tensor is a Kostant module.
\end{lem}

\begin{proof} This follows from the simplified algorithm since the weight diagram of a
bipartition with $d(\lambda) = 0$ looks like

\medskip
\begin{center}

%alpha
  \begin{tikzpicture}
 \draw (-5,0) -- (6,0);
\foreach \x in {1,2,3,4,5} %vee
     \draw (\x-.1, .2) -- (\x,0) -- (\x +.1, .2);
\foreach \x in {-4,-3,-2,-1,0} %wedge
     \draw (\x-.1, -.2) -- (\x,0) -- (\x +.1, -.2);
\foreach \x in {} %cross
     \draw (\x-.1, .1) -- (\x +.1, -.1) (\x-.1, -.1) -- (\x +.1, .1);
%\foreach \x in {1} %circle
     %\draw  node at (1,0) [fill=white,draw,circle,inner sep=0pt,minimum size=6pt]{};

\end{tikzpicture}
\medskip

\end{center} after removing the crosses and circles. Applying $\theta$ means specifying a vertex,
say $V$, and switching all free labels at vertices $\geq V$. This will not create
any neighbouring vertices labelled $\vee \wedge \vee \wedge$.
\end{proof}

\begin{cor} \label{properties-kostant} If $L(\mu)$ is an irreducible mixed tensor then:
\begin{enumerate} 
\item The Kazhdan-Lusztig polynomials are monomials: $p_{\lambda,\mu}(q) =
q^{l(\lambda,\mu)}$ for all $\lambda \leq \mu$ and $l(\lambda,\mu)$ as defined in \cite[5.2]{Brundan-Stroppel-2}. In particular $\sum_{i \geq 0} dim Ext^i (K(\lambda),L(\mu)) \leq 1$ for all $\lambda \in X^+$.
\item $L$ possesses a resolution by multiplicity free direct sums of Kac modules
(BGG-resolution).
\end{enumerate}
\end{cor}

\begin{proof} This are properties of Kostant weights \cite{Brundan-Stroppel-2}, lemma
7.2 and theorem 7.3. 
\end{proof}

\begin{remark} In fact Kostant modules are precisely the irreducible representations whose associated Kazhdan-Lusztig polynomials are monomials. Note also that this corollary is a generalization of \cite{Cheng-Kwon-Lam} who dealt with the case of covariant representations.
\end{remark}

%-------------------------------------------

\subsection{Tensor products.} \label{sec:tensor} We recall a formula for the decomposition $F_{m|n}(R(\lambda)) \otimes F_{m|n}(R(\mu))$ based on the decomposition in Deligne's category.

\textit{Caps.} We attach to the weight diagram of a bipartition a cap-diagram as in \cite{Brundan-Stroppel-1}. For integers $i<j$ one says that $(i,j)$ is a $\vee \wedge$-pair if they are joined by a cap. For $\lambda, \mu \in \Lambda$ one says that $\mu$ is linked to $\lambda$ if there exists an integer $k \geq 0$ and bipartitions $\nu^{(n)}$ for $0 \leq n \leq k$ such that $\nu^{(0)} = \lambda,  \nu^{(k)} = \mu$ and the weight diagramm of $\nu^{(n)}$ is obtained from the one of $\nu^{(n-1)}$ by swapping the labels of some pair $\vee \wedge$-pair. Then put \[ D_{\lambda, \mu}(m-n) = D_{\lambda, \mu} = \begin{cases} 1 \ \ \mu \text{ is linked to } \lambda \\ 0 \ \ \text{otherwise.} \end{cases} \] One has $D_{\lambda,\lambda} = 1$ for all  $\lambda$. Further $D_{\lambda,\mu} =  0$ unless $\mu = \lambda$ or $|\mu| = (|\lambda^L| - i, |\lambda^R| - i)$ for some $i> 0$.  Let $t$ be an indeterminate and $R_{\delta}$ respective $R_t$ the split Grothendieck rings of $\uRep(GL_{\delta})$ over $k$ respective of $\uRep(GL_t))$ over the fraction field $k(t)$. Now define $\lift_{\delta}:R_{\delta} \to R_t$ as the $\Z$-linear map defined by $\lift_{\delta}(\lambda) = \sum_{\mu} D_{\lambda,\mu} \mu$. By \cite[Theorem 6.2.3]{Comes-Wilson} $\lift_{\delta}$ is a ring isomorphism for every $\delta \in k$. 

%\bigskip

\textit{Tensor products.} To get the tensor product in $\calR_{m|n}$ the tensor product is computed in $\uRep(GL_d)$ and then pushed to $Rep(GL(m|n))$ by means of the tensor functor $F_{m|n}$. We identify elements of $R_{\delta}$ and $R_t$ with formal linear combinations of bipartitions. We also often use a $;$ to separate $\lambda^L$ and $\lambda^R$ in order to avoid brackets. For example $(i;1^j)$ denotes the element in $R_{\delta}$ corresponding to the indecomposable element $R((i), (1^j))$. By \cite[Theorem 7.1.1]{Comes-Wilson} the following decomposition holds for arbitrary bipartitions in $R_t$: \[ \lambda \mu = \sum_{v \in \Lambda} \Gamma_{\lambda \mu}^{\nu} \nu\] with the numbers \cite[Theorem 5.1.2]{Comes-Wilson} \[ \Gamma_{\lambda \mu}^{\nu} = \sum_{\alpha,\beta,\eta,\theta \in P} (\sum_{\kappa \in P} c_{\kappa \alpha}^{\lambda^L} c_{\kappa \beta}^{\mu^R}) \ (  \sum_{\gamma \in P} c_{\gamma \eta}^{\lambda^R} c_{\gamma \theta}^{\mu^L}) \ c_{\alpha \theta}^{\nu^L} c_{\beta \eta}^{\nu^R}, \] where the $c$-coefficients are the Littlewood-Richardson numbers. In particular if $\lambda \vdash (r,s)$, $\mu \vdash (r',s')$, then $\Gamma_{\lambda\mu}^{\nu} = 0$ unless $|\nu| \leq (r+r',s+s')$. As a special case we obtain \[ (\lambda^L; 0) \ (0;\mu^R) = \sum_{\nu} \sum_{\kappa \in P} c_{\kappa \nu^L}^{\lambda^L} c_{\kappa \nu^R}^{\mu^R} \nu \] in $R_t$. So to decompose tensor products in $\uRep(GL_{\delta})$ apply the following three steps: Determine the image of the lift $\lift_{\delta}(\lambda \mu)$ in $R_t$, use the formula above and then take $\lift_{\delta}^{-1}$.

\textit{Projective modules.} Note that $Proj$ is a tensor ideal in $\calR_{m|n}$, i.e. the tensor product of a projective module with any other module will split in a direct sum of projective modules. Since every projective module is in the image of $F_{m|n}$ and we have an explicit bijection $\theta$ between the projective modules and bipartitions with $k(\lambda) = n$,  the tensor product formula in the Deligne category gives us an  algorithm for the decomposition.

%\bigskip
%
%Note that $Proj$ is a tensor ideal, ie. the tensor product of a projective module with any other module will split in a direct sum of irreducible typical representations and projective covers of atypical modules \[ P \otimes M = \bigoplus P_i \oplus \bigoplus L(\lambda). \] Since every projective module is in the image of $F_{m|n}$ and we have an explicit bijection $\theta$ between the projective modules and bipartitions with $k(\lambda) = n$,  the tensor product formula in the Deligne category gives us an explicit algorithm for the decomposition.

\medskip

\begin{example} We compute the tensor product $P(1,1,1,0|0) \otimes P(1,1,1,0|0)$ in $Rep(GL(4|1))$. The corresponding bipartition is $\theta^{-1}(1,1,1,0|0) = (1^4;1)$. We have \[ \lift(1^4;1) = (1^4;1) \oplus (1^3;0).\] So we have to compute the tensor product \[ ((1^4;1) + (1^3;0)) \otimes ((1^4;1) + (1^3;0) \] in $R_t$. This decomposes in $R_t$ as \begin{align*} & (2^4;2) + (2^4;1^2) + ((2^3,1^2); 1^2) + 4((2^3,1);1) + 2(2^3;0)  + ((2^2,1^4);2) \\ & + ((2,1^4); 1^2) + 4((2^2,1^3);1) + 4((2^2,1^2);0) + ((2,1^6);2)  + ((2,1^6);1^2)\\ & + 4((2,1^5);1)  + 4((2,1^4);0) + (1^8;2) + (1^8;1^2) + 4 (1^7;1) + 4(1^6;0) \end{align*} and gives in $\calR_{4|1}$ the decomposition \begin{align*} & P(1,1,1,0|0) \otimes P(1,1,1,0|0) = \\ & P(2,2,2,1|-1) \oplus P(2,2,2,-1|1) \oplus 2P(2,2,2,0|0) \\  & \oplus L(2,2,1,-1|2) \oplus L(2,1,0,0|1) \oplus 4 L(2,2,1,0|1) \oplus 4 L(2,2,0,0|0) \\ & \oplus L(2,1,1,-1|3) \oplus L(2,1,0,0|3) \oplus 4 L(2,1,1,0|2) \oplus 4L(2,1,1,1|1) \\ &  \oplus L(1,1,1,-1|4)   \oplus L(1,1,0,0|4)   \oplus L(1,1,1,0|3) \oplus L(1,1,1,1|2).  \end{align*} 
\end{example}

%------------------------------------------

%-------------------------------------------

%-----------------------------------------------

\subsection{Duals} \label{duals} We also obtain a description of the dual of any irreducible module. Brundan \cite{Brundan-Kazhdan} gave an algorithm using certain operators on crystal graphs. For an algorithm on the cup diagram $\underline{\lambda}$ see \cite{Brundan-Stroppel-2}. Any irreducible module occurs as socle and head in its projective cover. Clearly $P(\lambda^{\dagger})^{\vee} = P((\lambda^{\dagger})^{\vee})$. On the other hand $P(\lambda^{\dagger})^{\vee} = R(\lambda^L, \lambda^R)^{\vee}  = R(\lambda^R, \lambda^L) = P((\lambda^{\dagger})^{\vee})$. So to compute the dual of an irreducible module, take its highest weight and associate to it the unique $(m|n)$-cross bipartition $(\lambda^L,\lambda^R)$ labelling its projective cover, switch it to $\tilde{\lambda} =(\lambda^R,\lambda^L)$ and then compute $\tilde{\lambda}^{\dagger}$. Then  $L(\lambda^{\dagger})^{\vee} = L(\tilde{\lambda}^{\dagger}).$ For an explicit description of the dual of an irreducible $GL(n|n)$-module using this see \cite{Heidersdorf-Weissauer-tensor}.

\begin{example} We compute the duals of the irreducible modules in the maximal atypical block of $\calR_2$. Since every such module is a Berezin-twist of one of the $S^i:=[i,0] = L(i,0|0,-i), \ i \in \N$, we may restrict to this case. The projective cover of $S^i$ is the module $R((i+1,1); (2,1^i))$. Hence the dual of the projective cover $P[i,0]$ is the module $R((2,1^i),(i+1,1))$. The irreducible module in the socle has weight $[1,1-i]$, hence \[ (S^i)^{\vee} = [1,1-i], \] i.e. $S^i = Ber^{i-1} (S^i)^{\vee}$. In particular the representations $ Ber^{-l} S^{2l+1}$ are selfdual.
\end{example}

\subsection{Contravariant modules for $m=n$.} The contravariant modules are the modules in the decomposition $T(0,r) = (V^{*})^{\otimes r}$. Hence they are the duals of the covariant modules $\{ \lambda \} = S_{\lambda}(V)$. Recall that the highest weight of $\{ \lambda\}= : L(\mu)$ is obtained as follows: Put $\mu_i = \lambda_i$ for $i=1,\ldots,n$ and $\mu_{n+i} = max(0, \lambda^*_i - n)$ for $i=1,\ldots,n$ where $\lambda^*$ is the conjugate partition and $\lambda$ is an $(n,n)$-hook partition \cite{Berele-Regev} \cite{Sergeev}. The set of this partitions is denoted by $H(n,n)$. Put further $(\lambda_1,\ldots, \lambda_r)^v = (-\lambda_r, \ldots,- \lambda_1)$. Recall that for $\lambda \in H(n,n)$ we have $\lambda^* \in H(n,n)$. %We give a closed formula for the highest weight of a contravariant module.

\begin{lem} $\{\lambda\}^{\vee}$ has highest weight $\mu^v$ where $\mu$ is the highest weight of $\{ \lambda^* \}$. 
\end{lem}

\begin{proof} We compute the weight diagram of $\{ \lambda \}^{\vee}$ in two different ways. The highest weight of $\{ \lambda \}^{\vee}$ is the highest weight in the socle of the mixed tensor $R(0,\lambda)$, and its weight diagram can be calculated using the description of $\theta$. Now we compute the highest weight of $\mu^v$. The highest weight of $\{\lambda\}$ is given by  $\mu_i = \lambda_i$ for $i=1,\ldots,n$ and $\mu_{n+i} = max(0, \lambda^{*}_i - n)$ for $i=1,\ldots,n$. For the transposed partition $\lambda^*_i = \# \{ \lambda_i \ | \ \lambda_i \geq i\}$. Hence the highest weight of $\lambda^{*}$ is given by $\mu_i = \lambda^{*}_i = \# \{ \lambda_i \ | \ \lambda_i \geq i\}$ for $i=1,\ldots,n$ and $\mu_{n+i} = max(0, \lambda_i - n)$ for $i=1,\ldots,n$. Applying $()^v$ yields the proposed highest weight of $\{ \lambda\}^{\vee}$ \begin{align*}  \mu =  (- max(0,\lambda_n - n), \ldots, - max(0, \lambda_1 - n) \ |  \  - \lambda_n^*, \ldots, - \lambda_1^*)).\end{align*} Now we determine $I_{\times}$ and $I_{\circ}$ of this weight according to the rules of Brundan-Stroppel. It is a tedious but elementary inspection to see that one obtains the same weight diagram in both cases.
\end{proof}

%\bigskip

%------------------------------

%----------------------------

%\medskip

%------------------------------------------------------------

%------------------------------------------------------------

\section{The Duflo-Serganova functor}\label{DS}

Let $M$ be a $\g = \mathfrak{gl}(m|n)$-module. For any $x \in X = \{x \in \g_1 \ | \ [x,x] = 0\}$ there exists $g \in GL(m) \times GL(n)$ and isotropic mutually orthogonal linearly independent roots $\alpha_1, \ldots, \alpha_k$ such that $Ad_{g}(x) = x_1 + \ldots + x_k$ with $x_i \in \g_{\alpha_i}$. The number $k$ is called the rank of $x$ \cite{Serganova-kw}.  For any $x$ of $rk(x) = k$ we have the cohomological tensor functor $M \mapsto F_x(M)$  from $\calR_{m|n} \to T_{m-k|n-k} = \calR_{m-k|n-k} \oplus \Pi \calR_{m-k|n-k}$ \cite{Duflo-Serganova} \cite{Serganova-kw} \cite{Heidersdorf-Weissauer-tensor}. We quote \cite[Theorem 2.1]{Serganova-kw} \cite[Corollary 2.2]{Serganova-kw}.

\begin{thm} If $at(M) < rk(x)$, then $F_x(M) = 0$. If $at(M) = rk(x)$, then $F_x(M)$ is a typical module. If $rk(x) = r$, then $at(F_x(M)) = at(M) - r$.
\end{thm}

From now on we will study the Duflo-Serganova tensor functor for special $x$. We define \[ x_r = \begin{pmatrix} 0 & \epsilon_r \\ 0 & 0 \end{pmatrix}, \ \ \epsilon_r = diag(1,\ldots,1,0,\ldots,0) \] with $r$ 1's on the diagonal. We denote the corresponding tensor functor by $DS_{x_r}$. If $r=1$ we simply write $DS$. An easy computation shows the next lemma.

\begin{lem} $DS_{x_r}$ maps $V$ to the standard representation of $GL(m-r|n-r)$.
\end{lem}

\begin{prop} Under $DS_{x_r}$ \begin{align*} R(\lambda) \mapsto \begin{cases} 0 & k(\lambda) > n-r \\ R(\lambda)  & \text{else. } \end{cases} \end{align*} In the case $r=1$ this specialises to \begin{align*} R(\lambda) \mapsto \begin{cases} 0  & R(\lambda) \ projective \\ R(\lambda) &  \text{else. } \end{cases}. \end{align*}  
\end{prop}

\begin{proof} This follows from the diagram \[ \xymatrix{ \uRep (GL_{m-n}) \ar[d]^{F_{m|n}} \ar[dr]^{F_{m-r|n-r}} & \\ T_{m|n} \ar[r]^{DS_{x_r}} & T_{m-r|n-r} }. \] Since $DS_{x_r}$ maps the standard representation to the standard representation, the universal property of Deligne's category implies that the diagram is commutative. In the case $r=1$ the kernel of $DS_{x}$ consists of the $(m-1|n-1)$-cross bipartitions which are not $(m|n)$-cross. This is equivalent to $k(\lambda) = n$ which is equivalent to $R(\lambda)$ projective. 
\end{proof}

\textbf{Remark.} This is a special case of a more general result \cite{BKN-complexity}: If $M$ is $^*$-invariant, then $M$ is projective if and only if $F_x(M) = 0$ for some $x$ of rank 1. 
\medskip

\textbf{Example.} If $M:= R((n-1,n-2,\ldots,1); (n-1,n-2,\ldots,1)^*)$ in $GL(n|n)$, then the socle has weight $(n-2,n-3,\ldots,1,0,0\ |\ 0,0,-1,\ldots,-n+2)$. We obtain $M_x = P(n-2,n-3,\ldots,1,0 \ | \ 0,-1,\ldots,-n+2)$ in $Rep(GL(n-1|n-1))$ for $x$ of rank 1.

\begin{lem} Let $y \in X$ be of rank $r$ such that $DS_y$ maps the standard representation to the standard representation. Then $DS_y = DS_{x_r}$ when restricted to $T$. 
\end{lem}

\begin{proof} This follows from the diagram above and the universal property of Deligne's category. 
\end{proof}

\begin{lem} If $R(\lambda)$ is irreducible, so is $DS_{x_r} (R(\lambda))$. 
\end{lem}

\begin{proof} $R(\lambda)$ is irreducible if and only if $d(\lambda) = 0$. The defect of a bipartition only depends on the difference $m-n = (m-r) - (n-r)$.
\end{proof}

%--------------------------------

%--------------------------------

%---------------------------------

\section{Irreducible representations in the image}

\begin{lem} Let $\Gamma$ be a block of atypicality $k \leq n$. Then $\Gamma$ contains a unique irreducible mixed tensor.
\end{lem}

\begin{proof} The block is characterized by the position of the $m-k$ crosses and $n-k$ circles on the number line. Denote by $L^{core}$ the typical $GL(m-k|n-k)$-module which is given by this position of the circles and crosses. Then $L^{core} = R(\lambda_{\Gamma})$ for a unique bipartition $\lambda_{\Gamma}$ of $rk(\lambda_{\Gamma}) = n-k$ and $d(\lambda_{\Gamma}) = 0$. This bipartition defines also an irreducible mixed tensor in $\Gamma \subset \calR_{m|n}$ since the weight diagram of a bipartition depends only on $m-n = m-k - (n-k)$. Assume that we would have two irreducible mixed tensors $R(\lambda_{\Gamma})$ and $R(\lambda'_{\Gamma})$ in $\Gamma$. Then both map to $L^{core}$ when applying $DS$ $k$ times or $DS_{x_k}$ one time. Since $DS(R(\lambda)) = R(\lambda)$ this implies  $R(\lambda_{\Gamma}) = R(\lambda'_{\Gamma})$.
\end{proof}

\begin{thm} Every Kostant module is a Berezin-twist of an irreducible mixed tensor.\label{Kostant-existence}
\end{thm}

%We will prove that this is also true in the maximally atypical case provided $m>n$ in section \ref{max-atyp-irred}.

\begin{proof} If $m=n$, the maximal atypical Kostant modules are the Berezin powers $Ber^r$, $r \in \Z$. Since $Ber^0 = \one$ is a mixed tensor, the assertion is clear in this case. We now exclude this case and describe the Berezin twist explicitely. We use the description how to obtain $\lambda^{\dagger}$ from $R(\lambda_{\Gamma})$ in the typical $GL(m-k|n-k)$-case from section \ref{sec:theta}. The highest weight $\lambda^{\dagger}$ of the mixed tensor $R(\lambda_{\Gamma})$ in $\calR_{m|n}$ is then obtained as follows from the weight diagram of the bipartition $\lambda$: If $M$ is the rightmost vertex labelled $\times$ or $\circ$ we distinguish the two cases a) $n-k \geq M$ or b) $M \geq n-k$. If $n-k \geq M$ we switch exactly the labels at vertices $> n-k$. In case b) we switch all labels at vertices $>M$ and the first $M-n + 2k$ free labels $<M$. If $L$ is a $k$-fold atypical Kostant module, consider the weight diagram of its highest weight. We denote by $\vee^{min}$ the vertex with the leftmost label $\vee$, by $z$ the number of crosses and circles at vertices $>  \vee^{min}$ and $<M$, and by $\vee^{max}$ the vertex with the leftmost label $\vee$. If $\vee^{max} > M$, we move $\vee^{max}$ with a Berezin-twist to the vertex $n-k$. If $\vee^{max} < M$ we move $\vee^{min}$ with a Berezin twist to the position $M - (M-n +2k) - z = n - 2k - z$. In both cases we get an irreducible mixed tensor. Indeed we find a typical $GL(m-k|n-k)$-module $L^{core'} = R(\lambda_{\Gamma'})$ with the same $\times$, $\circ$-labeling as $Ber^{\ldots} \otimes L$. By the rules of $\lambda_{\Gamma'} \mapsto \lambda_{\Gamma'}^{\dagger}$ we get $R(\lambda_{\Gamma'}) \simeq Ber^{\ldots} \otimes L \in \calR_{m|n}$.
\end{proof}

In particular we have now an algorithm to decompose the tensor product between any two Kostant-modules in $\calR_{m|n}$: Given two $\lambda,\mu$ Kostant weights we shift both into $T$ \begin{align*} L(\lambda) \otimes Ber^{\lambda'} = L(\tilde{\lambda})  \in T, \ \ L(\mu) \otimes Ber^{\mu'} = L(\tilde{\mu})  \in T \end{align*} where $\lambda', \ \mu'$ are described in the proof of \ref{Kostant-existence} . Therefore \begin{align*} L(\lambda) \otimes L(\mu) & = ( L(\tilde{\lambda}) \otimes L(\tilde{\mu})) \ \otimes \ (Ber^{\lambda'} \otimes Ber^{\mu'}) \\ & = \bigoplus_{\nu} d_{\tilde{\lambda}, \tilde{\mu}}^{\nu} L(\nu) \otimes Ber^{\lambda' + \mu'}\end{align*} for certain coefficients $d_{\tilde{\lambda}, \tilde{\mu}}^{\nu}$ which can be calculated explicitely as in section \ref{sec:tensor}. In particular the tensor product of two such modules can be decomposed explicitely (see an example below).

\begin{example} The irreducible module with weight $(6,4,2,1,1,0|-2,-2,-2,-2)$ is a 3-fold atypical Kostant module in $\calR_{6|4}$. Twisting with $B^{-1}$  gives the mixed tensor $R(\lambda_{\Gamma'}) =R((5,3,1);5)$.  
\end{example}

\begin{example} The highest weight $\mu = (12,12,10,10,10,10,0|-11,-11,-12)$ of $GL(7|3)$ is maximal atypical with rightmost $\vee$ at position 8 and two crosses at position 11 and 12 to the right. Now twist $L(\mu)$ with $Ber^{-10}$ to move $\vee$ to position -2 and obtain $\tilde{\mu} = (2,2,0,0,0,0,-10|-1,-1,-2)$. We get $Ber^{-10} \otimes L(\mu) = R(2,2; 13,1)$.
\end{example}

\begin{example} It is an easy exercise to check that the $(n-1)$-times atypical irreducible mixed tensors in $\calR_{n}$ are the \[ R(i; 1^j), \ i \geq 0, \ j \neq i, \ (i,j) \neq (0,0) \text{ and their duals } R(1^j; i).\] For $\lambda  = (i; 1^j)$ we get $R(\lambda) = L(\lambda^{\dagger}) = L(i,0,\ldots,0|0,\ldots,0,-j)$.
\end{example}

\subsection{Character and dimension formula.}\label{character-dimension} By Comes and Wilson \cite{Comes-Wilson}, thm 8.5.2, we have a character and dimension formula for mixed tensors which is much easier than the general formulas of \cite{Su-Zhang}. By loc.cit the character of a mixed tensor is given as \[ ch \ R(\lambda) = \sum_{\mu} D_{\lambda,\mu}(m-n) s_{\mu} \] where the sum runs over the bipartition's occurring in $\lift(\lambda)$ and $s_{\mu}$ is the composite supersymmetric Schur polynomial associated to $\lambda$. Given an arbitrary Kostant module $L(\lambda)$ (which is not a Berezin power) and the unique Berezin-twist $Ber^{r}$ with $Ber^{r} \otimes L(\mu) = R(\lambda_{\Gamma'})$, the character of $L(\lambda)$ is \begin{align*} ch \ L(\lambda) & = ch \ Ber^{-r} \cdot \ ch R(\lambda_{\Gamma'}) \\ & = ch \ Ber^{-r} \cdot \ s_{\lambda_{\Gamma'}}. \end{align*} A similar formula has been obtained before in \cite{Chmutov-Hoyt-Reif}. Since the dimension does not change after tensoring with $Ber^{-r}$we get \[ dim \ L(\lambda) = dim \ R(\lambda_{\Gamma'}) = d_{\lambda_{\Gamma'}}.\]

\subsection{The case $GL(m|1)$}

In the case $GL(m|1)$ and $SL(m|1)$ every weight is a Kostant weight. Since $Ber$ is trivial in the $SL$-case we obtain:

\begin{cor} Up to a twist of a suitable power of $Ber$ every irreducible module of $GL(m|1)$ is in $T$. Every irreducible module of $SL(m|1)$ is in $T$.
\end{cor}

%\textbf{Example}: The $GL(2|1)$-case. Since $l(\lambda) \leq 1$, the irreducible atypical mixed tensors are the covariant and contravariant tensors. The highest weight $(\lambda_1,\lambda_2| \lambda_3)$ is atypical if and only if either $\lambda_2 = - \lambda_3$ or $\lambda_3 = - \lambda_1 - 1$. The covariant module $R(a;0)$ has highest weight $(a,0|0)$ and the contravariant module $R(0;b)$ has highest weight $(0,-b+1|-1)$. The modules with highest weights $(\lambda_1,\lambda_2|-\lambda_2)$ are Berezin twists of covariant modules and the modules with highest weights $(\lambda_1,\lambda_2| - \lambda_1 - 1)$ are Berezin twists of contravariant modules.
%
%\bigskip 

By \cite{Germoni-sl} \cite{Goetz-Quella-Schomerus} the indecomposable modules in $\calR_{m|1}$ are the (Anti-)ZigZag-modules $Z^r(\lambda)$, $\bar{Z}^r(\lambda)$ and the projective hulls of the irreducible atypical representations. 

\begin{cor} If $l(\lambda) \leq m-1$, $R(\lambda^{\dagger})$ is irreducible singly atypical. If $l(\lambda) > m - 1$ and $d(\lambda)  = 0$ then $R(\lambda ) = L(\lambda^{\dagger})$ is typical. If $\lambda$ is any bipartition with $d(\lambda) = 1$ then $R(\lambda) = P(\lambda^{\dagger})$. 
\end{cor} %uses the l(\lambda) \leq m-1 -> max atypical from below

\begin{cor} In the decomposition $L(\lambda) \otimes L(\mu)$ between two irreducible $GL(m|1)$-modules no (Anti)ZigZag module $Z^l(a)$ with $l \geq 2$ appears as a direct summand. 
\end{cor}

Since any irreducible $GL(m|1)$-module is up to an explicit Berezin-Twist in $T$, the tensor product formula in Deligne's category and the description of the image of $F_{m|1}$ solves the problem of decomposing the tensor product of any two irreducible $GL(m|1)$-representations. 

\medskip

\textbf{Example.} We compute $L(2,0,0,0|0) \otimes L(1,0,0,0|-1)$ in $\calR_{4|1}$. Applying $\theta^{-1}$ we see that the corresponding bipartitions are $(2;0)$ and $(1;1)$. Since the defect is zero, we only have to compute $(2;0) \otimes (1;1)$ in $R_3$. By \cite{Comes-Wilson}, example 7.1.3, we have \[ (2;0) \otimes (1;1) = ((2,1);1) + (3;1) + (1^2;0) + (2;0) \] for $\delta = 3$ in $R_{\delta}$. Hence \begin{align*} & L(2,0,0,0|0)   \otimes   L(1,0,0,0|-1) = \\  & \ \ \  \  L(1,1,0,0|0) \oplus L(2,0,0,0|0) \oplus L(3,0,0,0|-1) \oplus L(2,1,0,0|-1)\end{align*} in $Rep(GL(4|1))$.

\medskip

ZigZag modules \cite{Germoni-sl} \cite{Goetz-Quella-Schomerus} of length greater than 1 never occur in the image of $F_{m|1}$. However the tensor product between an indecomposable projective module with a ZigZag-module is easily reduced to the known cases by the following well-known fact:

\begin{prop} Let $P$ be projective and $M$ any module. Then $ P \otimes M = \bigoplus_i P \otimes M_i$ where the sum runs over the composition factors $M_i$ of $M$. 
\end{prop}

\begin{proof} Use induction on the length of $M$. If $M$ is of length $n$ consider a sequence \[ \xymatrix{ 0 \ar[r] & M_i \ar[r] & M \ar[r] & M' \ar[r] & 0 } \] with $length(M') = n-1$. Tensoring with $P$ and using that $Proj$ is a tensor ideal we see that the sequence splits.  
\end{proof}

\begin{cor} Let $P$ be an indecomposable projective $GL(m|1)$-Module. Then \begin{align*} P \otimes Z^r(a)  = \bigoplus_{a_i}  P \otimes L(a_i), \ \ P \otimes \overline{Z}^r(a)  = \bigoplus_{a_i}  P \otimes L(a_i) \end{align*} where the sums run over the composition factors $L(a_i)$ of $Z^r(a)$ respectively $\overline{Z}^r(a)$.
\end{cor}

All in all the only remaining unknown tensor products in the $GL(m|1)$-case are the tensor products $Z^r(a) \otimes Z^s(b)$, ($r$ or $s$ $\geq 2$) and vice versa for the Anti-ZigZag-modules. If $r,s$ are odd their tensor product decomposes as given by the Littlewood-Richardson Rule for $GL(m-n)$ modulo contributions of superdimension 0 \cite{Heidersdorf-semisimple}.

\subsection{Irreducible maximal atypical mixed tensors for $m >n$} For two weights $\lambda = (\lambda_1, \ldots,\lambda_m \ | \ \lambda_{m+1}, \ldots, \lambda_{m+n})$ and $\mu = (\mu_1,\ldots, \mu_m \ | \ \mu_{m+1},\ldots, \mu_{m+n})$ say that $\lambda \succeq \mu$ if there exists $i \in \{1,\ldots,m\}$ with the property $\lambda_j = \mu_j$ for all $j<i$ and $\lambda_i > \mu_i$. We will see in lemma \ref{degree} that $\{ \lambda^L\} \otimes \{\lambda^R\}^{\vee} = R(\lambda) \oplus \bigoplus R(\mu^j)$ with $|\mu^j| < |\lambda|$.

\begin{lem}  Let $R(\lambda)$ be maximally atypical irreducible. Then $R(\lambda) = L(\lambda^{\dagger})$ with $L(\lambda^{\dagger}) \succ L(\mu_j^{\dagger})$ for all $j$.
\end{lem}

\begin{proof} Define $I_x^{max}(\lambda) =$ largest vertex with a $\times$ or $\vee$. Then the statement amounts to show that $I_x^{max}(\lambda) \geq I_x^{max}(\mu_j)$ for all $j$. This is a combinatorial exercise using our description of $\theta$.
\end{proof}

Hence the maximally atypical $R(\lambda)$ for $m>n$ of $sdim \neq 0$ could be characterized as follows: Take all the tensor products of two $(m|n)$-Hook partitions $\lambda^L, \lambda^R$ such that $(\lambda^L,\lambda^R)$ is $(m|n)$-cross. Then the $R(\lambda)$ are the indecomposable modules in the decomposition $\{\lambda^L \} \otimes \{\lambda^R\}^{\vee}$ which satisfy $R(\lambda) = L(\lambda^{\dagger}) \succ L(\mu_j^{\dagger})$ for all $j$. %I don't know a similar characterisation for the other $R(\lambda)$.

%\bigskip

%-----------------------------------

%--------------------------------

%----------------------------------------

%-------------------------------------------

\section{Elementary properties of the $R(\lambda)$}

Given two $(m|n)$-Hook partitions $\lambda^L, \lambda^R$ we form the bipartition $(\lambda^L,\lambda^R)$. It is in general not $(m|n)$-cross. 

\begin{lem} \label{degree} Let $\lambda^L$, $\lambda^R$ be two $ (m|n)$-Hook partitions such that $\lambda = (\lambda^L,\lambda^R)$ is $(m|n)$-cross. Then $\{\lambda^L\} \otimes \{\lambda^R\}^{\vee}$ contains $R(\lambda)$ as a direct summand. In the decomposition \[ \{\lambda^L\} \otimes \{\lambda^R\}^{\vee} = R(\lambda) \oplus \bigoplus R(\mu^j) \] all $\mu^j$ satisfy $(\mu^j)^L_i \leq \lambda^L_i$ and $(\mu^j)^R_i \leq \lambda^R_i$ for all $i$ and $|\mu^j| < |\lambda|$.
\end{lem}

\begin{proof} 
In $R_t$  \[ (\lambda^L,0) \otimes (0,\lambda^R) = \sum_{\nu} \sum_{\kappa \in P} c_{\kappa,\nu^L}^{\lambda^L} c_{\kappa,\nu^R}^{\lambda^{R}} \ \nu. \] Putting $\kappa = 0$ yields $\nu^L = \lambda^L, \ \nu^R = \lambda^R$. Hence \[ (\lambda^L,0) \otimes (0,\lambda^R) =  (\lambda^L,\lambda^R) +  \sum_{\nu} \sum_{\kappa \in P, \kappa \neq 0} c_{\kappa,\nu^L}^{\lambda^L} c_{\kappa,\nu^R}^{\lambda^{R}} \ \nu. \] All other bipartitions $\nu = (\nu^L, \nu^R)$ will have degree stricty smaller than $(\lambda^L,\lambda^R)$ and length $\geq$ than $l(\lambda)$. By Comes-Wilson $\lift_{d}(\lambda) = \lambda + \ldots$ where the other bipartitions are obtained by swapping successively $\vee \wedge$-pairs, i.e. decreasing the coefficients of the bipartition. Since $\lambda$ is the largest bipartition, $R(\lambda)$ will occur with multiplicity one in the decomposition. 
\end{proof}

For any two partitions $\lambda^L, \lambda^R$ such that the pair $(\lambda^L,\lambda^R)$ is $(m|n)$-cross we define \[ \A_{\lambda^L \lambda^R} := \{\lambda^L\} \otimes \{\lambda^R\}^{\vee}.\] Recall from section \ref{sec:repr} that $()^*$ denotes the graded dual.                                                                                                                           

\begin{prop} $R(\lambda^L,\lambda^R)$ is $^*$-invariant 
\end{prop}

\begin{proof} Clearly $\A_{\lambda^L\lambda^R}$ is $^*$-invariant since irreducible modules are $^*$-invariant. In the decomposition \[ \A_{\lambda^L\lambda^R} = R(\lambda^L,\lambda^R) \oplus \bigoplus_i R(\mu^j) \] $R(\lambda^L,\lambda^R)$ occurs as a direct summand with multiplicity 1; and $|\lambda| > |\mu^j|$. Assume $R(\lambda)$ would not be $^*$-invariant. Then there exists a $\mu^j$ occuring with multiplicity 1 in the decomposition with $R(\lambda)^* = R(\mu^j)$. Write $\mu^j = ((\mu^j)^L, (\mu^j)^R)$. As for $\lambda$, $R(\mu^j)$ occurs with multiplicity 1 in the decomposition of the $^*$-invariant \[ \A_{(\mu^j)^L,(\mu^j)^R} = R(\mu^j) \oplus \bigoplus_j R(\nu^j)\]  with degree strictler larger then the other bipartitions $\nu^j$. Hence there exists a $\nu^j$ with $R(\mu^j)^* = R(\nu^j)$. Since applying $^*$ is the identity, this forces $\nu^j = \lambda$. However $|\lambda| > |\mu^j| > |\nu^j|$, a contradiction.  
\end{proof}

\begin{lem} If $R(\lambda)$ is maximally atypical then $d(\lambda) \geq d(\mu^j)$ for all $j$. If $R(\lambda)$ is maximally atypical and irreducible then $\{\lambda^L\} \otimes \{ \lambda^R\}^{\vee}$ is completely reducible and splits into maximally atypical irreducible summands.
\end{lem}

\begin{proof} $R(\lambda)$ is maximally atypical if and only if $rk(\lambda) = 0$. Hence $k(\lambda) \geq k(\mu^j)$ implies the first statement. If $R(\lambda)$ is additionally irreducible, then $d(\mu^j) = 0$ for all $j$. 
\end{proof}

%------------------------------------------

\begin{lem} In the tensor product \[ R(\lambda) \otimes R(\mu) = \sum_i \kappa_{\lambda \mu}^{\nu^i} R(\nu^i)\] all $\nu_i$ satisfy \[ k(\nu^i) \geq max (k(\lambda), k(\mu)).\]
\end{lem}

\begin{proof} Let $n' =  max (k(\lambda), k(\mu))$. Apply $DS_{n-n'}: \calR_{m|n} \to \calR_{m'|n'} \oplus  \Pi \calR_{m'|n'}$. Without loss of generalisation $n' = k(\lambda)$. Then $R(\lambda)$ is projective in $\calR_{m'|n'}$. The projective modules form an ideal, hence $R(\lambda ) \otimes R(\mu)$ decomposes in $\calR_{m'|n'}$ into indecomposable projective modules. Since the tensor product comes from the Deligne category \[ \xymatrix{ & \uRep(GL_{m-n}) \ar[dr]^{F_{m'|n'}} \ar[dl]_{F_{m|n}} & \\ \calR_{m|n} \ar[rr]^{DS_{n-n'}} & & \calR_{m'|n'} \oplus \Pi \calR_{m'|n'} }\] we have in $\calR_{m|n}$ \[ \sum_i \kappa_{\lambda \mu}^{\nu_i} R(\nu_i) \oplus ker(DS_{n-n'}) \] with $k(\nu^i) \geq n'$ for all $i$. Further $ker(DS_{n-n'})$ are the mixed tensors $R(\gamma)$ with $n'< k(\gamma) \leq n$.
\end{proof}

\begin{example} Any irreducible summand in $R(\lambda) \otimes R(\mu)$ has atypicality $\leq n - max(k(\lambda),k(\mu))$.
\end{example}

We denote by $T^i_{mn}$ the subset of mixed tensors with $k(\lambda) \geq i$. Recall that a set of objects $\mathcal{I}$ in a tensor category $\mathcal{C}$ is an ideal if $X\otimes Y\in \mathcal{I}$ whenever $X\in \mathcal{C}$ and $Y\in\mathcal{I}$.

\begin{cor} The $T^i$ are ideals in $T$. We have strict inclusion \[ T^0 \supsetneq T^1 \supsetneq \ldots \supsetneq T^n \] with $T^0 =T$ and $T^n = Proj$.
\end{cor}

By \cite{Serganova-kw} any two irreducible objects of atypicality $k$ generate the same ideal in $\calR_{m|n}$. Therefore write $I_k$ for the ideal generated by an irreducible object of atypicality $k$. Clearly $I_0 = Proj$ and $I_n = T_n$ since it contains the identity. This gives the following filtration of $\calR$ \[Proj = I_0 \subsetneq I_1 \subsetneq \ldots I_{n-1} \subsetneq I_n = \calR_{m|n}\] with strict inclusions by \cite{Serganova-kw} and \cite{Kujawa-generalized-kac-wakimoto}. 

\begin{lem} $I_k|_T = T^{n-k}$ for all $k=0,\ldots,n$. 
\end{lem}

\begin{proof} For any atypicality $k$ there exists an irreducible mixed tensor with that atypicality, hence $I_k|_T \subset T^{n-k}$. Conversely let $R(\lambda) \in T^{n-k}$. It occurs as a direct summand in $R(\lambda^L,0) \otimes R(0,\lambda^R)$. Then $max(k(\lambda^L,0), k(0,\lambda^R)) \leq n-k$, hence $rk(\lambda^L,0), rk(0,\lambda^R) \leq n-k$, hence $at(R(\lambda^L,0), at(R(0,\lambda^R)) \geq k$, hence $R \in I_l$ for any $l \geq k$.
\end{proof}

\subsection{The constituent of highest weight}

We have seen that the irreducible modules in $T$ are the ones with $d(\lambda) = 0$. We describe the constituent of highest weight of $R(\lambda)$ for $d(\lambda) > 0$. The constituents of $R(\lambda)$ are given by  $[R(\lambda)]=  [G_{\Delta \Gamma}^t L(\zeta)] = \sum_{ \gamma \subset \beta \rightarrow^t \zeta, \ red(\underline{\gamma}t) = \underline{\zeta} } [ L(\gamma)]$. The condition $\gamma \subset \beta$ implies $\beta \geq \gamma$ in the Bruhat order, hence the constituent of highest weight must be among the $\beta \to^t \zeta$. We define $A_{\lambda}$ by taking the weight diagram of $\lambda^{\dagger}$ and by labelling all caps in the matching $t$ by $\wedge \vee$. This is the maximal element in the Bruhat order among all the possible $\beta$. It will give the constituent of highest weight if $A_{\lambda}$ satisfies the condition $red(\underline{A_{\lambda}}t) = \underline{\zeta}$. This is an easy combinatorial exercise.

\begin{lem} $A_{\lambda}$ is the constituent of highest weight of $R(\lambda)$. It occurs with multiplicity 1 in the middle Loewy layer. 
\end{lem}

\begin{cor}\label{projective-k-0} Two direct sums $\bigoplus P_i$, $\bigoplus Q_j$ of projective modules are equal in $\calR_{m|n}$ if and only if they are equal in $K_0$.
\end{cor}

\begin{proof} It suffices to test this for a single block $\Gamma$. It is easy to see that $A_{\lambda}$ and $R(\lambda) = P(\lambda^{\dagger})$ determine each other. Hence it is equivalent to give the direct sum $\bigoplus_{i \in I}P_i$ in $K_0$ or the set $ \{ A_{i} \}_{i \in I}$. Hence $\bigoplus P_i = \bigoplus Q_j$ if and only if $\{ A_i \}_{i \in I} = \{ A_j \}_{j \in J}$. We are done if we can determine the set $\{ A_i \} $ uniquely from the decomposition $[\bigoplus P_i]$ in $K_0$. We will give an algorithm to do so. The block is represented by the numberline with $k$ $\vee$'s (with variable position) and $m-k$ $\times$ and $n-k$ $\circ$ (with fixed position). Let $P$ be the set of composition factors of $\bigoplus P_i$. It may be identified with the set of the corresponding weight diagrams. We go from the right to the left through these diagrams. Let $i_1$ be the rightmost position with a $\vee$ in $P$. We restrict to the subset $P_{i_1}$ of $P$ of diagrams with a $\vee$ at position $i_1$. From $i_1$ we move to the left. Let $i_2$ be the next position with a $\vee$ among the diagrams in $P_{i_1}$. Let $P_{i_1 i_2}$ the set of weight diagrams with a $\vee$ at position $i_1$ and $i_2$.  Iterating this procedure we obtain $P_{i_1 i_2 \ldots i_k}$. This set consists of the weight diagram of a unique weight, possibly with multiplicity $\geq 1$ (since $\times$, $\circ$ and $\vee$'s are fixed). We claim that this weight is of the form $A_{i}$ for some $P_i$. This is clear: The weight determines a composition factor $L(...)$ of some $P(a)$. If $L(...) \neq A_{a}$, then $A_{a} > L(...)$ in contradiction to the construction above. The factor $A_{i}$ determines the corresponding projective module $P_i$. We remove all the composition factors of the copies of $P_i$ from $P$. Now we apply the same algorithm again to the set $P \setminus r[P_i]$ to obtain again a weight of the form $A_{l}$ with corresponding projective module $P_l$. We remove its composition factors etc until there are no weights left in $P$. Hence we have constructed all the weights $A_{i}$ from the $K_0$-decomposition.    
\end{proof}

%----------------------------------------------------

%-------------------------------------------------------------

%----------------------------------------------------

%-------------------------------------------------------------

%----------------------------------------------------

%-------------------------------------------------------------

%\input{mixed-tensors-ch2-atypical-mixed-tensors}  %Maximally atypical mixed tensors

%\part{Maximally atypical modules in the space of mixed tensors}\label{sec:atypical}

%-----------------------------------------------------------

\section{Multiplicities and tensor quotients}\label{maximal-atypical-m>n}

By the universal property of Deligne's category there exists for $\delta = d \in \N$ a full tensor functor $F_d: \uRep(GL_d) \to Rep(GL(d))$. Given a bipartition $\lambda = (\lambda^L,\lambda^R)$ of length $\leq d$, $\lambda^L = (\lambda^L_1,\ldots,\lambda^L_s,0,\ldots), \ \lambda_s^L > 0$, $\lambda^R = (\lambda^R_1,\ldots,\lambda^L_t,0,\ldots), \ \lambda_t^R > 0$, put \[ wt(\lambda) = \lambda_1^L \epsilon_1 + \ldots + \lambda^L_s \epsilon_s - \lambda^R_t \epsilon_{d+1-t} - \ldots - \lambda_1^R \epsilon_d.\] This defines the irreducible $GL(d)$-module $L(wt(\lambda))$ with highest weight $wt(\lambda)$. By \cite{Comes-Wilson} \begin{align*} F_d(R(\lambda)) = \begin{cases} L(wt(\lambda)) & l(\lambda) \leq d \\ 0 & l(\lambda) > d. \end{cases} \end{align*} This defines a bijection between bipartitions of length $\leq d$ with highest weights of $GL(d)$. For $d=m-n>0$ we have the two tensor functors \[ \xymatrix{ & \uRep(GL_{m-n}) \ar[dr]^{F_{m-n}} \ar[dl]_{F_{m|n}} & \\ \calR_{m|n} & & Rep(GL(m-n))} \] given by mapping the standard representation to the two standard representations (for $m=n$, $GL(m-n)$ is the empty group). Another tensor functor is the following: By \cite{Weissauer-gl} there exists a purely transcendental field extension $K/k$ of transcendence degree $n$ and a $K$-linear weakly exact tensor functor \[ \rho: \calR_{m|n} \otimes_k K \to Rep  (GL(m-n)) \otimes svec_K.\] By \cite{Weissauer-gl} each simple maximal atypical object $L(\mu)$ maps to the isotypic representation $\Pi^{p(\mu)} m(\mu) \rho(V)$ where $m(\mu)$ is a positive integer, $V$ is the ground state (see loc.cit) of the block of $\mu$ and $p(\mu)$ is the parity of $\mu$. After a suitable specialisation of $\rho$ we may assume that $\rho$ is defined over $k$ and maps the standard to the standard representation. Hence we get the commutative diagramm of tensor functors (due to Deligne's universal property)  \[ \xymatrix{ & \uRep(GL_{m-n}) \ar[dl]_{F_{m|n}} \ar[dd]^{F_{m-n}\otimes svec} \\ Rep(GL(m|n)) \ar[dr]^{\rho} & \\ & Rep(GL(m-n)) \otimes svec. } \] Here the functor $F_{m-n}\otimes svec$ maps $R(\lambda)$ to the even representation \[ L(wt(\lambda)) \in Rep(GL(m-n)) \subset Rep(GL(m-n)) \otimes svec.\]

\begin{lem} $R(\lambda)$ has superdimension $\neq 0$ if and only if $l(\lambda) \leq m-n$.
\end{lem}
 
\begin{proof} This follows from the commutative diagram above. Use the bijection between the highest weights of $GL(d)$ and bipartitions of length $\leq m-n$ to choose for any $(m|n)$-cross bipartition $\lambda$ the irreducible highest weight module $L(wt(\lambda))$. By the commutativity the indecomposable module $R(\lambda)$ has to map to $L(wt(\lambda))$. Its superdimension is the dimension of $L(wt(\lambda))$.
\end{proof}

The mixed tensors form a pseudoabelian tensor subcategory $T$ of $\calR_{m|n}$. It is closed under duals ($T(r,s)^{\vee} = T(s,r)$) and contains the identity. The functor of Weissauer \[ \rho: \calR_{m|n} \to Rep(GL(m-n))\otimes svec \] can be restricted to $T$. Let us denote by $\NN$ the tensor ideal of negligible morphisms \cite{Heidersdorf-semisimple} \cite{Andre-Kahn}.

\begin{thm}\label{tensor-equivalence} The functor $\rho_T : T \to Rep(GL(m-n)) \otimes svec $ factorises over $T/\NN$ and defines an equivalence of tensor categories \[ T/\NN \simeq Rep(GL(m-n)).\] It maps the element $R(\lambda)$ to the irreducible element $L(wt(\lambda))$.
\end{thm}

\begin{proof} The functor will factorize if $\rho_T$ is full \cite{Heidersdorf-semisimple}. This follows from the commutative diagram since an indecomposable module maps to an irreducible module. $R(\lambda) \mapsto L(wt(\lambda))$ is forced by the commutativity of the diagram. By the bijection between highest weights of $GL(m-n)$ and bipartitions of lenght $\leq m-n$ the functor is one-to-one on objects. Fully faithful follows from Schur's lemma in the semisimple tensor category $T/\NN$. \end{proof}

%--------------------------------------------------------

%-----------------------------------------------------------

%
%\subsection{An alternative approach} Assume $m>n$. All bipartitions are $(m|n)$-cross. We provide an alternative proof that $\rho: T/\NN \simeq Rep(GL(m-n))$ which does not use the existence of a tensor functor $Rep(GL(m|n)) \to Rep(GL(m-n)) \otimes svec$.

%\bigskip

Theorem \ref{tensor-equivalence} plays a crucial role in \cite{Heidersdorf-semisimple}. An analogous theorem holds in the orthosymplectic case \cite{Comes-Heidersdorf}. The following proposition is a simple combinatorial exercise.

\begin{prop} Let $\lambda$ be a bipartition of length $\leq m-n$. Then $d(\lambda) = 0$. In particular the maximal atypical $R(\lambda)$ with $sdim R(\lambda) \neq 0$ are irreducible.
\end{prop}

The equivalence $T/\NN \simeq Rep(GL(m-n))$ means that for $\lambda,  \mu$ bipartitions of lenght $\leq m-n$, the tensor product  $R(\lambda) \otimes R(\mu)$ is given by the Littlewood-Richardson rule for $GL(m-n)$ up to superdimension 0.

\begin{cor} Let $\lambda$ and $\mu$ be such that $l(\lambda) + l(\mu) \leq m-n$. Then $R(\lambda) \otimes R(\mu)$ splits completely into irreducible maximally atypical modules. The decomposition rule is given by the Littlewood-Richardson rule for $GL(m-n)$.
\end{cor}

\subsection{An alternative approach} We provide an alternative proof that $T/\NN \simeq Rep(GL(m-n))$ which does not use the existence of a tensor functor $Rep(GL(m|n)) \to Rep(GL(m-n)) \otimes svec$. Assume we have proven that $l(\lambda) \leq m-n$ implies $d(\lambda) = 0$ and therefore that $R(\lambda)$ is irreducible in $Rep(GL(m|n))$.

\begin{lem} Let $\lambda,  \nu$ be bipartitions of lenght $\leq m-n$. Then their tensor product is given by the Littlewood-Richardson rule for $GL(m-n)$ up to superdimension 0. More precisely \[ R(\lambda) \otimes R(\mu) = \bigoplus_{\nu, \ l(\nu) \leq m-n} c_{wt(\lambda),wt(\mu)}^{wt(\nu)} R(\nu) \ mod \ \NN \] where $c_{wt(\lambda),wt(\mu)}^{wt(\nu)}$ denotes the multiplicity of the $GL(m-n)$-representation $L(wt(\nu))$ in the decomposition $L(wt(\lambda)) \otimes L(wt(\mu))$.
\end{lem}

\begin{proof} (cf. the proof of 7.1.1 in \cite{Comes-Wilson}) Let $\nu_1, \ldots \nu_k$ be bipartitions such that \[ \lambda \mu = \nu_1 + \ldots \nu_k\] in $R_t$. Since $\lift(\lambda) = \lambda, \ \lift(\mu) = \mu$ we may assume mod $\NN$ that all $\nu_i$ have length $\leq m-n = d$. So $d$ satisfies $d \geq l(\nu_i)$ for all $i$ and $\lift_d$ fixes $\lambda, \mu, \nu_1, \ldots \nu_k$. Hence $ \lambda \mu = \nu_1 + \ldots \nu_k$ holds in $R_d$ as well. Using the tensor functor $F_d: \uRep(GL_d) \to Rep(GL(d))$ which maps $R(\lambda)$ to $L(wt(\lambda))$ we obtain \begin{align*} L(wt(\lambda)) \otimes L(wt(\mu)) & = L(wt(\nu_1)) \oplus \ldots \oplus L(wt(\nu_k)) \\ & = \bigoplus_{\nu, l(\nu) \leq m-n} c_{wt(\lambda),wt(\mu)}^{wt(\nu)} R(\nu) \end{align*} by the Littlewood-Richardson rule in $Rep(GL(d))$. Taking the preimage one obtains modulo $\NN$ the result.
\end{proof}

It is now easy to recover the equivalence $T/\NN \simeq Rep(GL(m-n))$ . Since \[F_{m|n}: \uRep (GL_{m-n}) \to \calR_{m|n}\] has its image in $T$ we can consider the diagram \[ \xymatrix{ & \uRep(GL_{m-n}) \ar[dl]^{F_{m|n}} \ar[dd]^{F_{m-n}} \\ T \ar[d] & \\ T/\NN \ar@{.>}[r] & Rep(GL(m-n)). } \] Using the bijection between the irreducible elements $R(\lambda)$ and the irreducible elements in $Rep(GL(m-n))$,
we define the lower horizontal functor by putting $R(\lambda) \mapsto L(wt(\lambda))$ on objects. Since both categories are semisimple tensor categories, Schur's lemma holds and the functor sends the morphism $id:R(\lambda) \to R(\lambda)$ to $id:L(wt(\lambda)) \to L(wt(\lambda))$. The results on the tensor products show that this defines a tensor functor. It is clearly fully faithful.

%\bigskip

%-----------------------------------------------

%------------------------------------------------

------------------------------------------------------

%--------------------------------------------------------------------------

%--------------------------------------------------------------------------------

%-------------------------------------------------------

%----------------------------------------------------------

%---------------------------------------

%---------------------------------------

\section{Maximally atypical $R(\lambda)$ for $m =n$}\label{maximal-atypical-m=n}

For $m=n$ no nontrivial maximally atypical irreducible modules are in $T$ because their superdimension does not vanish \cite{Serganova-kw}, \cite{Weissauer-gl}. Assume from now on that $\lambda^{\dagger}$ is in the maximal atypical block $\Gamma$, i.e. the weight diagram has no $\times$, no $\circ$ and exactly $n$ $\vee$'s. Recall that a maximal atypical irreducible representation is irreducible if and only if its weight is of the form $(\lambda_1,\ldots,\lambda_n \ |  \ - \lambda_n, \ldots, - \lambda_1)$. In this case we write $[\lambda_1,\ldots,\lambda_n]$ for this representation.

\begin{prop}\label{max-atypical} $R(\lambda^L,\lambda^R)$ is maximal atypical if and only if $\lambda^R = (\lambda^L)^*$. 
\end{prop}

\begin{proof} Since there are no $\circ$ and no $\times$ \[ I_{\vee} \cup I_{\wedge} = \Z, \ \  I_{\vee} \cap I_{\wedge} = \emptyset.\] Hence $\lambda^L$ and $\lambda^R$ determine each other uniquely. It is easy to see that the conjugate partition satisfies these conditions.

\end{proof}

\begin{cor} $T(r,s)$ contains a maximally atypical summand only for $r=s$.
\end{cor}

\begin{proof} By \cite{Brundan-Stroppel-5} and the characterisation of maximally atypical $R(\lambda) $ \[ pr_{\Gamma} \ T(r,s) = \bigoplus R(\lambda, \lambda^*) \] where $|\lambda| =r-t, \ |\lambda^{*}| = s-t$. Since $|\lambda| = |\lambda^{*}|$ this can only happen for $r=s$. 
\end{proof}

\textbf{Notation.} From now on we always write $R(\lambda)$ where $\lambda$ is a partition such that $(\lambda, \lambda^{*})$ is $(n|n)$-cross.

%------------------------------------------------------

\subsection{The involution $I$} Recall that the dual of an indecomposable element in $T$ is given by $R(\lambda^L, \lambda^R)^{\vee} = R(\lambda^R,\lambda^L)$. Similarly we define \[ I R(\lambda^L, \lambda^R) := R(({\lambda^R})^*, ({\lambda^L})^*).\] 

\begin{lem} This is a well-defined operation on $T$ for $m=n$ (ie. $(({\lambda^R})^*, ({\lambda^L})^*)$ is again $(n|n)$-cross). $I$ is an involution and commutes with Tannaka duality. $I$ is the identity if and only if $R(\lambda)$ is maximally atypical.
\end{lem}

\begin{proof} Let $i \in {1,\ldots,n}$ have the property $\lambda_{i+1}^L + \lambda^R_{n-i+1} \leq n$, so $\lambda_{i+1}^L \leq k$ and $\lambda^R_{n-i+1} \leq n-k$ for some $k$. Then $(\lambda_{k+1}^L)^* \leq i$ and $(\lambda^R_{n-k+1})^* \leq n-i$, hence $(\lambda^L_{k+1})^*  + (\lambda_{n-k+1}^R)^* \leq n$. The other statements are clear.
\end{proof}

\begin{remark} For $m>n$ the bipartition $((\lambda^R)^*, (\lambda^L)^*)$ may fail to be $(m|n)$-cross.
\end{remark}

\begin{lem}  $I$ preserves dimensions.
 \end{lem}

\begin{proof} Since the dimension is preserved under dualising $(\lambda^L, \lambda^R) \mapsto (\lambda^R,\lambda^L)$, we only have to take care of $(\lambda^L,\lambda^R) \mapsto ((\lambda^L)^*),(\lambda^R)^*))$. By \cite[(43)]{Comes-Wilson} \[ dim R(\lambda) = \sum_{\mu \subset \lambda} D_{\lambda,\mu} d_{\mu} \] where $d_{\mu}$ is obtained from the composite supersymmetric Schur polynomial $s_{\mu}(\underline{x},\underline{y})$, $\underline{x} = (x_1,\ldots,x_n), \ \underline{y} = (y_1, \ldots,y_n)$ by setting $x_i = 1 = y_i$ for all $i=1,\ldots,n$. By \cite[2.39]{Moens} $s_{\mu}(x|y) = s_{\mu^*}(y|x)$, hence $d_{\mu} = d_{\mu^*}$. Let $\lambda \vdash (r,s)$. Then $\lambda^* \vdash (r,s)$. By \cite{Cox-deVisscher} the number $D_{\lambda \mu}$ is the decomposition number $[\Delta_{r,s}(\lambda): L_{r,s}(\mu)]$ where $\Delta_{r,s}$ is a cell module for the walled Brauer algebra $B_{rs}$. It is clear that $D_{\lambda\mu} = D_{\lambda^*,\mu^*}$, hence if $\sum_{\mu \subset \lambda} D_{\lambda \mu} d_{\mu} = d_{\mu_1} + \ldots + d_{\mu_r}$, then $\sum_{\mu' \subset \lambda^*} D_{\lambda^* \mu'} d_{\mu'} = d_{\mu_1^*} + \ldots d_{\mu_r^*}$.
\end{proof}

%\medskip
\begin{example} $I( (i;1^j)) = (j;1^i)$, hence $\lambda^{\dagger} = [i,0|0,-j]$ and $I\lambda^{\dagger} = [j,0|0,-i]$.
\end{example}

%----------------------------------------------------------

%-----------------------------------------------------------

\section{Existence of maximal atypical mixed tensors for $m=n$}\label{existence}

Our aim in the next sections is to obtain some information about the tensor products of maximal atypical irreducible modules by embedding them either in the socle or as the constituent of highest weight (in the Bruhat order) in a mixed tensor. Of course every irreducible representation occurs in the socle and top of its projective cover, but it is preferable for the applications in section \ref{sec:tensor-products} to work with modules of smaller Loewy length. Let $[\lambda] = [\lambda_1, \ldots, \lambda_n]$ be maximally atypical in $\calR_n$. We normalize $[\lambda]$ so that $\lambda_n = 0$. More generally a weight with $\lambda_n \geq 0$ will be called positive. If $k \in \{1,\ldots,n\}$ is the biggest index with $\lambda_k \neq 0$ we say that the weight is of length $k$. Such a weight defines a partition $\lambda$ of length $k$. 

%\bigskip

\begin{lem} If $l(\lambda) \leq k$, then $d(\lambda) \leq k$. If $\lambda_1 \leq k$, then $d(\lambda) \leq k$. In particular a (maximal atypical) mixed tensor can be projective only if $l(\lambda) \geq n$ and $\lambda_1 \geq n$. 
\end{lem}

\begin{example} There is a unique projective mixed tensor $R(\lambda)$ with $\lambda$ of smallest degree. It is given by $\lambda = (n,n-1,\ldots,1)$ and gives the projective cover of $[\lambda^{\dagger}] = [n-1,n-2,\ldots,1,0]$.
\end{example}

\begin{thm} \label{thm:existence} 1) For every positive weight $\lambda^{\dagger} = [\lambda^{\dagger}_1,\ldots,\lambda^{\dagger}_n]$ of length $k$ exists a unique mixed tensor $R(\lambda)$ of defect $k$ and length$(\lambda^L) = k$ and $socle (R(\lambda)) = [\lambda^{\dagger}]$. 2) For every positive weight $[\lambda]$ of length $k$ the mixed tensor $R(\lambda)$ has defect $\leq k$ and contains $[\lambda]$ with multiplicity 1 in the middle Loewy layer. $[\lambda]$ is the constituent of highest weight in $R(\lambda)$.
\end{thm}

In particular $[\lambda] \to R(\lambda)$ gives a bijection between the positive weights of length $k$ and mixed tensors given by partitions of length $k$. 

\begin{proof} Proof of 1). We construct $R(\lambda)$ explicitly. To an irreducible highest weight we associate its cup diagram with $n$ cups. Since the length of $[\lambda^{\dagger}]$ is $k$, exactly $k$ $\wedge$'s are are bound in a cup with a $\vee$ associated to one of the $\lambda^{\dagger}_1,\ldots,\lambda^{\dagger}_k$. Label the $k$ $\wedge$'s from the rightmost to the leftmost position by $\{ v_1, v_2, \ldots, v_k\}$. Then define the partition $\lambda = (v_1,v_2+1, v+3 +2, \ldots,v_k +k -1)$. Then $l(\lambda) = d(\lambda) = k$. The $k$ cups of $\lambda$ agree with the $k$ cups of $[\lambda^{\dagger}]$ associated to the nontrivial $\lambda_i^{\dagger}$. By construction (and the positivity of $[\lambda^{\dagger}]$) the largest label in a cup in the cup diagram of $\lambda$ is at a vertex $\geq k$. We obtain the highest weight in $R(\lambda)$ according to the rules of section \ref{sec:theta} by switching all labels at vertices $\geq v_1$ and the first $v_1 +n - 2k$ labels at vertices $\leq v_1$ which are not part of a cup. Since the leftmost cup has its leftmost label at a vertex $\geq 2-k$, this means switching exactly all the free labels at vertices $\geq -n+1$. This switches all labels $\vee$ at vertices $\leq v_1$. Since the length is $k$ we have $\wedge's$ at all vertices $\-k$. The $n-k$ rightmost $\wedge$'s at the positions $-n+1, -n+2,\ldots, -n+(n-k)$ will be switched to $\vee$'s. These $n-k$ $\wedge$'s give the $n-k$ zeros in $\lambda^{\dagger}$. Uniqueness: We apply $DS' := DS^{n-k}: \calR_n \to R_k$. Then $DS'(R(\lambda))$ is the projective cover of a unique irreducible module and the mixed tensors of defect $k$ are in bijection with the projective covers of irreducible modules. We show $soc(DS'(R(\lambda))) = [\lambda_1^{\dagger},\ldots,\lambda_k^{\dagger}]$ which implies our assertions about the uniqueness of the mixed tensor of length $k$ with prescribed socle. Since $\lambda$ does not depend on $n$ we get the same weight and cup diagram of $\lambda$. The highest weight of the socle is obtained as above by switching all labels at vertices $\geq V_1$ and the first $v_1 +k - 2k$ labels at vertices $\leq v_1$ which are not part of a cup. This means that we do not switch the $n-k$ leftmost labels at the vertices $-n+1, -n+2,\ldots, -n+(n-k)$.

Proof of 2). We use 1). We start with the partition $\lambda$. We have seen that $l(\lambda) \leq n$ means switching the free $\wedge$'s at vertices $\geq -n+1$ and $\leq max(k(\lambda),M)$ to $\wedge$'s and vice versa. Similarly all the free vertices $\geq max(k(\lambda),M)$ are labelled by $\vee$'s which are switched to $\wedge$'s. We obtain $A_{\lambda}$ from $[\lambda^{\dagger}]$ by interchanging the $\vee$'s with the $\wedge$'s in the $d(\lambda)$ cups of $\underline{\lambda}$. Hence to get the weight diagram of $A_{\lambda}$ from the weight diagram of $\lambda$ means switching all labels at vertices $\geq -n + 1$. The $\wedge$'s at the vertices $\geq -n+1$ are at the vertices $\lambda_1,\lambda_2-1,\ldots,\lambda_k - k + 1, -k,\ldots,-n+1$, hence the $\vee$'s in the weight diagram of $A_{\lambda}$ are at the vertices $\lambda_1,\lambda_2-1,\ldots,\lambda_k - k + 1, -k,\ldots,-n+1$.
\end{proof}

%\textbf{Remark}: Note that different partitions of the same defect but different length can give the same highest weight in the socle. 

%\bigskip

\begin{example} Assume $[\lambda^{\dagger}] = [\lambda_1^{\dagger},\ldots,\lambda_k^{\dagger},0,\ldots,0]$ with $\lambda^{\dagger}_1 > \lambda^{\dagger}_2 > \ldots > \lambda^{\dagger}_k$. The $k$ $\wedge$'s in the $k$ cups are at the vertices $\lambda^{\dagger}_1+1,\lambda_2^{\dagger}, \lambda^{\dagger}_3 - 1,\ldots, \lambda^{\dagger}_k -k +2$, hence $\lambda = (\lambda^{\dagger}_1+1,\lambda_2^{\dagger}, \lambda^{\dagger}_3 - 1,\ldots, \lambda^{\dagger}_k -k +2)$.
\end{example}

\begin{remark} If the $\wedge$'s of $[\lambda^{\dagger}]$ in the $k$ cups are at the vertices $v_1,\ldots, v_k$ so that $\lambda = (v_1,v_2+1,\ldots,v_k +k-1)$, then $A_{\lambda} = [v_1,v_2+1,\ldots,v_k +k -1,0,\ldots,0]$. 
\end{remark}

\begin{lem} \label{lem:positivity} Let $\lambda$ be an $(n|n)$-cross partition. The socle of $R(\lambda)$ is positive if and only if $l(\lambda) \leq n$. The highest weight constituent $A_{\lambda}$ is positive if $l(\lambda) \leq n$.
\end{lem}

\begin{proof} This is an exercise using the algorithm for $\theta$ in section \ref{sec:theta}.
\end{proof}

We define the degree $deg[\lambda]$ of an arbitrary maximally atypical highest weight as $\sum_{i=1}^n \lambda_i$. With this definition the constituent of highest weight in $R(\lambda)$ is the constituent of largest degree. If $[\lambda]$ is positive, it defines a partition $\lambda$ and $deg[\lambda] = |\lambda|$.

\begin{lem} \label{lem: deg-estimate} We have $degA_{\lambda} \leq |\lambda|$ with equality if and only if $l(\lambda) \leq n$.
\end{lem}

\begin{proof} If $l(\lambda) \leq n$ we have seen this in \ref{thm:existence}. If $l(\lambda) >n$, then for the $n$ $\vee$ in the weight diagram of $[\lambda]$ there are at least $n$ $\wedge$'s $\{ \wedge_1,\ldots,\wedge_n\}$ corresponding to $n$ non-trivial $\lambda_i$ $\{\lambda_{i_1},\ldots,\lambda_{i_n}\}$ at vertices greater or equal to the vertices of the $n$ $\vee$.  Then $degA_{\lambda} \leq \sum_{i=1}^n \lambda_{i_j}$. Since the length of $\lambda$ is larger than $n$, $|\lambda| > \sum_{i=1}^n \lambda_{i_j}$.
\end{proof} 

\begin{example} i) The mixed tensors with $soc(R(\lambda)) = B^k, \ k \neq 0$, are the projective covers $P(B^k)$. We have \begin{align*} P(B^k) & = R((n+k)^n) \quad & k \geq 1 \\ P(B^{n-r}) & = R(n^r) & r>n. \end{align*} The mixed tensors with socle $\one$ are the modules  $R(k^k), \  k \in \{1,\ldots,n\}$. We remark that the constituent of highest weight in $R(k^k)$ is $[k,\ldots,k,0,\ldots,0]$ and the constituent of highest weight in $R((n+k)^n)$ is $[n+k,n+k,\ldots,n+k]$ for $k \in [-n+1,\infty)$.

ii) The mixed tensors $R(\lambda)$ with $A_{\lambda} = B^{\ldots}$ are the projective covers $P(B^k)$ with constituent of highest weight $B^{k+n}$ and the $R(k^n)$ for $1 < k < n$ with highest weight constituent $B^k$ and defect $k$.

\end{example}

%----------------------------------------------------------

%-----------------------------------------------------------

\section{The symmetric and alternating powers}\label{symmetric-powers}

We specialise the results of the previous section now to the case $d(\lambda) = 1$. We define \[ \A_{S^i}: = R(i;1^i) = R(i) \text{ and } \A_{\Lambda^i}:= (\A_{S^i})^{\vee} = R(1^i;i) = R(1^i).\]

\begin{lem} If $d(\lambda) = 1$, then $R(\lambda) = \A_{S^i}$ or $\A_{\Lambda^i}$ for some $i>0$.
\end{lem}

\begin{proof} For $d(\lambda) = 1$ there can be at most one jump $\lambda_j > \lambda_{j+1}$ in the partition, hence $\lambda = (a,0,\ldots)$ for some $a$ or $\lambda = (\underbrace{b,b,\ldots,b}_{a \ times},0,\ldots)$ for $a>1$. For $b>1$ two $\vee$ will occur, hence $d(\lambda) > 1$. 
\end{proof}

\begin{lem} $\A_{S^i} = R(i;0) \otimes R(0;1^i)$ and $\A_{\Lambda^j} =  R(1^j;0) \otimes R(0;j)$.  
\end{lem}

\begin{proof} We want to compute $((i);0) \otimes (0; 1^i)$ in $R_t$, hence the sum $\sum_{\nu} \sum_{\kappa \in P} c_{\kappa,\nu^L}^{(i)} c_{\kappa,\nu^R}^{(1^i)}$, hence we search the pairs $(\kappa,\nu), \ (\kappa,\nu^{*})$ such that $c_{\kappa,\nu}^{\lambda}$ respectively $c_{\kappa, \nu^{*}}^{\lambda^{*}} \neq 0$. The Pieri rules tells one that the only such pairs are the pairs \[ ((0), (i)) \longleftrightarrow ((0), (1^i)) \text{ and } \ ((1), (i-1)) \longleftrightarrow ((1), (1^{i-1})).\] Hence \[ (i;0) \otimes (0;1^i) = (i) \oplus (i-1).\] in $R_t$. Now clearly $\lift(i) = (i) \oplus (i-1)$, hence the result.
\end{proof}

We define $S^i = [i,0,\ldots,0]$ for integers $i \geq 1$.

\begin{lem} The Loewy structure of the $\A_{S^i}$ is given by ($n \geq 2$) \begin{align*} \A_{S^1} & = (\one, S^1, \one) \\ \A_{S^i} & = (S^{i-1}, S^i \oplus S^{i-2}, S^{i-1}) \quad 1<i \neq n \\ \A_{S^n} & = (S^{n-1}, S^{n} \oplus S^{n-2} \oplus B^{-1}, S^{n-1}). \end{align*}
\end{lem}

%\textbf{Remark}: For $n=1$ we get $\A = P(\one)$.

\begin{proof} We sketch the computation for $\A_{S^i}, \ 1 < i < m$. The module in the socle can be computed by applying $\theta$. The matching $t$ looks schematically like (picture for $i=4$)

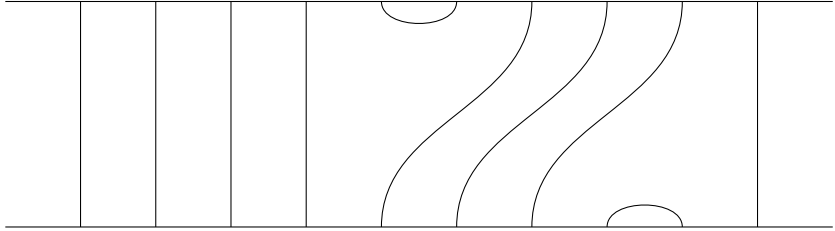
\begin{center}
\medskip

%alpha
  \begin{tikzpicture}
 \draw (-5,0) -- (6,0);
\foreach \x in {} %vee
     \draw (\x-.1, .2) -- (\x,0) -- (\x +.1, .2);
\foreach \x in {} %wedge
     \draw (\x-.1, -.2) -- (\x,0) -- (\x +.1, -.2);
\foreach \x in {} %cross
     \draw (\x-.1, .1) -- (\x +.1, -.1) (\x-.1, -.1) -- (\x +.1, .1);

\begin{scope} [yshift = -3 cm]

%lambda
 \draw (-5,0) -- (6,0);
\foreach \x in {} %vee
     \draw (\x-.1, .2) -- (\x,0) -- (\x +.1, .2);
\foreach \x in {} %wedge
     \draw (\x-.1, -.2) -- (\x,0) -- (\x +.1, -.2);
\foreach \x in {} %cross
     \draw (\x-.1, .1) -- (\x +.1, -.1) (\x-.1, -.1) -- (\x +.1, .1);
\end{scope}

%Verbindungen
\draw [-,black,out=270,in=90](-4,0) to (-4,-3);
\draw [-,black,out=270,in=90](-3,0) to (-3,-3);
\draw [-,black,out=270,in=90](-2,0) to (-2,-3);
\draw [-,black,out=270,in=90](-1,0) to (-1,-3);
\draw [-,black,out=270,in=90](2,0) to (0,-3);
\draw [-,black,out=270,in=90](3,0) to (1,-3);
\draw [-,black,out=270,in=90](4,0) to (2,-3);
\draw [-,black,out=270,in=90](5,0) to (5,-3);

%caps,cups
\draw [-,black,out=270, in=270](0,0) to (1,0);
\draw [-,black,out=90, in=90](3,-3) to (4,-3);

\end{tikzpicture}

\medskip
\end{center}

with the upper cup at the vertices $(0,1)$ and the lower one at the vertices $(i-1,i)$. To determine the remaining composition factors we search the $\gamma$ with $ \gamma \subset \beta \rightarrow^t \zeta, \ red(\underline{\gamma}t) = \underline{\zeta}$. Since $t$ and $\zeta$ are fixed and the matching has to be consistently oriented this determines $\beta$ up to the position at the unique cup in $t$ at position $(i-1,i)$. Now consider $\gamma$ where $\gamma$ is obtained from $\lambda^{\dagger} = S^{i-1}$ by moving the $\vee$ at position $i-1$ to position $i-2$. This gives a cup at position $(i-2,i-1)$. The lower reduction property is satisfied and gives the weight $S^{i-2}$. No other $ \gamma \subset \lambda^{\dagger}$ fulfill the summation conditions. The second possible case for $\beta$ (switching the $\wedge$ with the $\vee$ in the rightmost cup, hence moving $\vee$ one to the right) gives the module $[S^i] = [i,0,\ldots,0]$. As in the case of $\beta = \lambda^{\dagger}$ a second $\gamma \subset [S^i]$ may be obtained by moving the rightmost $\vee$ one to the left. The corresponding module is $[S^{i-1}]$ and gives the second copy of $[S^{i-1}]$. One can check that no other weight diagrams fulfill the summation conditions. The Loewy layers can be determined from the number of lower circles in $red(\underline{\gamma}t) = \underline{\one}$. The remaining cases can be treated in the same way. 
\end{proof}

\begin{remark} We obtain a recursive algorithm to compute the tensor product $L(v) \otimes S^i$ where $L(v)$ is a typical module in the $m=n$-case. The tensor product $L(v) \otimes \A$ (where $\A = \A_{S^1}$) is known since both modules are in the image of $F_{n|n}$. Since $L(v)$ is projective and $\A = (\one, S^1, \one)$, it splits into $2 L(v) \oplus L(v) \otimes S^1$. Removing the two $L(v)$ we obtain $L(v) \otimes S^1$. Similarly $L(v) \otimes \A_{S^1} = L(\nu) \otimes S^2 \oplus 2 L(\nu) \otimes S^1 \oplus L(v)$ which gives a formula for $L(v) \otimes S^2$. Iterating this procedure gives the decomposition of $L(v) \otimes S^i$ for any $i$ since a projective representation is determined by its composition factors by theorem \ref{projective-k-0}. In particular it gives an algorithm to decompose $L(v) \otimes L[a,b]$ where $L(v)$ is a typical $GL(2|2)$-module  and $L[a,b]$ is a maximal atypical representation of $GL(2|2)$. For the $psl(2|2)$-case see also \cite{Goetz-Quella-Schomerus-psl}.
\end{remark}

\begin{example} We want to compute $L(2,2|1,1) \otimes L(2,1|-1,-2)$ in $\calR_2$. We have $L(2,2|1,1) = R(2^3;0)$, so we compute $(2^3;0) \otimes ( (1;1) + (0;0)$ in $R_t$. This gives us \begin{align*} & L(2,2|1,1)  \otimes \A = \\ & \ \ L(3,1|4,3) \oplus L(3,2|1,0) \oplus L(2,1|3,1) \oplus L(2,2|2,0) \oplus 2 L(2,2|1,1). \end{align*} Removing $ 2L(2,2|1,1)$ we get the decomposition of $L(2,2|1,1) \otimes S^1$ and after twisting with $B$ we get \begin{align*} & L(2,2|1,1)  \otimes L(2,1|-1,-2) = \\ & \ \ L(4,2|3,2) \oplus L(4,3|0,-1) \oplus L(3,2|1,1) \oplus L(3,2|2,0) \oplus L(1,1|1,-1).\end{align*}
\end{example}

%------------------------------------------------------

%---------------------------------------------------------

\begin{remark} For $m>n$ the tensor product $R(i,0) \otimes R(0,1^i)$ splits into a sum of two irreducible modules \[ R(i,0) \otimes R(0,1^i) = R(i;1^i) \oplus R(i-1,1^{i-1}). \] In particular the adjoint representation $V \otimes V^{\vee}$ decomposes as \[ V \otimes V^{\vee} = \one \oplus L(1,0,\ldots,0 \ | \ 0,\ldots,0,-1).\] 
\end{remark}

%----------------------------------------------------------

%-------------------------------------------------

\section{Remarks on tensor products}\label{sec:tensor-products} 

We can embed any positive $[\lambda]$ of length $k$ in the socle of a mixed tensor of defect $k$ or as highest weight constituent. In the $S^i \otimes S^j$-case this permits us to obtain the decomposition of $S^i \otimes S^j$ \cite{Heidersdorf-Weissauer-gl-2-2} using a closed formula for the $\A_{S^i} \otimes \A_{S^j}$ tensor product. Copying the approach in the $\A_{S^i} \otimes \A_{S^j}$-case seems to be hopeless because general $R(\lambda)$ have lots of composition factors which are difficult to determine. We content ourselves with the following observations. The estimate in proposition \ref{comp} on the composition factors could be obtained from a restriction to $GL(n) \times GL(n)$, but the approach for the proof shows more and is also needed for \ref{proj-covers}.

\subsection{Composition factors.} As before we consider only bipartitions of the form $(\nu,\nu^*)$ and and we identify such a bipartition with the partition $\nu$.

\begin{lem} $\Gamma_{\lambda \mu}^{\nu}$ is zero unless $l(\nu) \leq l(\lambda) + l(\mu)$.
\end{lem}

\begin{proof} The Littlewood-Richardson coefficients $c_{\lambda \mu}^{\nu}$ are zero unless $l(\nu) \leq l(\lambda) + l(\mu)$ and $l(\nu) \geq max(l(\lambda), l(\mu))$. In the sum \[ \sum_{\kappa \in P} c_{\kappa \alpha}^{\lambda^L} c_{\kappa \beta}^{\mu^*} \] $c_{\kappa \alpha}^{\lambda^L} = 0$ unless $l(\alpha ) \leq l(\lambda^L)$. Similarly $c_{\gamma \theta}^{\mu^L} = 0$ unless $l(\theta) \leq l(\mu^L)$. Hence any $\nu^L$ with non-vanishing $c_{\alpha \theta}^{\nu^L}$ satisfies $l(\nu^L) \leq l(\alpha) + \l(\theta) \leq l(\lambda^L) + \l(\mu^L)$.
\end{proof}

\begin{lem} If $\Gamma_{\lambda \mu}^{\nu} \neq 0$, then $\nu_1 \leq (\lambda + \mu)_1$.
\end{lem}

\begin{proof} Follows at once from the corresponding property of the $c_{\lambda \mu}^{\nu}$.
\end{proof}

\begin{lem} If $c_{\lambda \mu}^{\nu} \neq 0$, then $\Gamma_{\lambda \mu}^{\nu} = (c_{\lambda \mu}^{\nu})^2$. These $\nu$ are exactly the $\nu$ with degree $|\lambda| + |\mu|$. If $\nu$ is any other partition with $\Gamma_{\lambda \mu}^{\nu} \neq 0$, then $|\nu| < |\lambda| + |\mu|$.
\end{lem}     

\begin{proof} We get $\Gamma_{\lambda \mu}^{\nu} = (c_{\lambda \mu}^{\nu})^2$ by putting $\kappa = 0$ and $\gamma = 0$ in the expression for $\Gamma_{\lambda,\mu}^{\nu}$: If $\kappa = 0$, then $\alpha = \lambda^L$ and $\beta = \mu^R$. If $\gamma = 0$, then $\eta = \lambda^R$ and $\theta = \mu^L$. Then $\Gamma_{\lambda,\mu}^{\nu} = c_{\lambda^L \lambda^R}^{\nu^L}  c_{(\lambda^L)^* (\lambda^R)^*}^{\nu^R}$. Since $c_{\lambda \mu}^{\nu} =  c_{\lambda^* \mu^*}^{\nu^*}$ we get  $\Gamma_{\lambda \mu}^{\nu} = (c_{\lambda \mu}^{\nu})^2$ if we only consider maximally atypical contributions $\nu = (\nu,\nu^*)$. In general $c_{\lambda \mu}^{\nu} \neq 0$ implies $|\lambda| + |\mu| = |\mu|$, hence \begin{align*} |\nu^L| & = |\lambda| +  |\mu| - |\kappa| - |\gamma| \\ |\nu^R| & = |\lambda| +  |\mu| - |\kappa| - |\gamma|. \end{align*} and for non-trivial $\kappa$ or $\gamma$ the partition $\nu$ cannot satisfy $c_{\lambda \mu}^{\nu} \neq 0$.
\end{proof}

\textbf{Notation.} We call any $\nu$ with $c_{\lambda \mu}^{\nu} \neq 0$ a classical solution of $\lambda \mu$ or $\Gamma_{\lambda \mu}^{\nu}$.

\begin{lem} In $R_0$ \[ R(\lambda) \otimes R(\mu) = \bigoplus_{|\nu| = |\lambda| + |\mu|} (c_{\lambda \mu}^{\nu})^2 R(\nu) \oplus \bigoplus_{|\nu| < |\lambda| + |\mu|} d_{\lambda \mu}^{\nu} R(\nu)\] for some coefficients $d_{\lambda \mu}^{\nu}$. 
\end{lem}

\begin{proof} To calculate the tensor product in $R_t$ we have to compute $\lift(\lambda) \otimes \lift(\mu)$ in $R_t$. Now $\lift(\lambda)  = \lambda + \sum_i \lambda^i$ with partitions $\lambda^i$ of degree strictly smaller then the degree of $\lambda$. Likewise for $\mu$. Hence the partitions $\nu$ of maximal degree cannot occur in any other tensor product from the partitions obtained from $\lift(\lambda)$ respectively $\lift(\mu)$ other then $\lambda \otimes \mu$. To pass from $R_t$ to $R_0$ we have to take $\lift^{-1}$ of the tensor product. Since the lift is strictly degree decreasing, none of the partitions $\nu$ can occur in in the lift of another partition.
\end{proof}

Note that in general a classical solution $\nu$ will not be $(n|n)$-cross. Hence in $\calR_n$ the sum above only incorporates $\nu$ which are $(n|n)$-cross. However the mixed tensor $R(\lambda + \mu)$ occurs always in the decomposition $R(\lambda) \otimes R(\mu)$ in $\calR_n$ due to following lemma.

\begin{lem} If $l(\lambda) \leq n$, then $\lambda$ is $(n|n)$-cross. 
\end{lem}

\begin{proof} A bipartition $\lambda$ is $(n|n)$-cross if and only if at least one of inequalities $\lambda_i + \lambda_{n+2-i}^* \leq n$ for $i=1,\ldots, n+1$ is satisfied. If $l(\lambda) \leq n$, then $\lambda_1^* \leq n$, hence $\lambda_{n+1} + \lambda_1^* \leq n$.
\end{proof}

\begin{lem} \label{lem: degree} Let $\nu$ be a classical solution of length $\leq n$ in $R(\lambda) \otimes R(\mu)$. Let $[\nu']$ be a constituent in $R(\lambda) \otimes R(\mu)$. Then $deg[\nu'] \leq deg \ A_{\nu}$ with equality if and only if $[\nu'] = A_{\nu'}$ with $\nu'$ a classical solution of length $\leq n$.
\end{lem}

\begin{proof} This follows from the degree estimates in section \ref{existence}.\end{proof} 

\begin{prop} \label{comp} Assume $\nu$ is a classical solution with $l(\nu) \leq n$. Then $[\lambda] \otimes [\mu]$ contains the composition factor $[\nu]$ with multiplicity $(c_{\lambda \mu}^{\nu})^2$.
\end{prop}

\begin{proof} We know that $[\lambda]$ and $[\mu]$ are the constituents of highest weight in $R(\lambda)$ and $R(\mu)$. Let \begin{align*} [R(\lambda)]  = \sum_i [\lambda^i] + [\lambda], \ \  [R(\mu)]  = \sum_j [\mu^j] + [\mu] \end{align*} with $[\lambda] > [\lambda^i]$ and $[\mu] > [\mu^j]$ for all $i,j$ in the Bruhat order. Assume first that all $[\lambda^i]$ and $[\mu^j]$ are positive. Then they define the mixed tensors $R(\lambda^i)$ and $R(\mu^j)$ and $|\lambda| > |\lambda^j|$ and $|\mu| > |\mu^j|$ for all $i,j$. Accordingly none of the mixed tensors $R(\nu)$ with $|\nu| = |\lambda| + |\mu|$ can appear in a tensor product \[ R(\lambda) \otimes R(\mu^j), \ R(\lambda^i) \otimes R(\mu), \ R(\lambda^i) \otimes R(\mu^j).\] Now \begin{align*} [R(\lambda] \otimes [R(\mu)] & = [\lambda] \otimes [\mu] \\ & +  \sum_j [\lambda] \otimes [\mu^j] + \sum_i [\lambda^i] \otimes [\mu] + \sum_{i,j} [\lambda^i] \otimes [\mu^j] \end{align*} and similarly for $R(\lambda) \otimes R(\mu^j), \ R(\lambda^i) \otimes R(\mu)$ and $R(\lambda^i) \otimes R(\mu^j)$. We claim: Since the $R(\nu)$ do not appear in any of these tensor products, their constituent of highest weights does not appear as a composition factor in any of these tensor products. If this claim is true, none of the tensor products $[\lambda] \otimes [\mu^j], \ [\lambda^i] \otimes [\mu]$ and $[\lambda^i] \otimes [\mu^j]$ can contain $[\nu]$ as a composition factor, hence $[\nu]$ must be a composition factor of $[\lambda] \otimes [\mu]$. For the proof of the claim we distinguish two cases. Since $l(\nu) \leq n$, $[\nu]$ is positive. Consider first a summand $R(\theta)$ with $l(\theta) \leq n$. Then $[\theta]$ is positive and $\sum \theta_i < \sum \nu_i$. Since $[\theta]$ is the constituent of highest weight of $R(\theta)$, all constituents are smaller then $[\nu]$. If $R(\theta)$ is a summand with $l(\theta) > n$, $deg \ A_{\theta} < deg[\nu]$. This proofs the claim. Finally we remove the assumption that all $[\lambda^i]$ and $[\mu^j]$ are positive. If $[\lambda^i]$ is not positive we twist with $Ber^{- \lambda^i_n}$. Similarly for $[\mu^j]$. Call these modules $[\tilde{\lambda}]$ respectively $[\tilde{\mu}]$. Then $deg(\tilde{\lambda}) = deg(\lambda^i) + n (-1) \lambda^i_n$ and $deg(\tilde{\mu}) = deg(\mu^j) + n (-1) \mu^j_n$. Embed the modules $[\tilde{\lambda}]$ respectively $[\tilde{\mu}]$ as constituents of highest weight in $R(\tilde{\lambda})$ respectively $R(\tilde{\mu})$ as in \ref{thm:existence}. In $R(\tilde{\lambda}) \otimes R(\tilde{\mu}) = \bigoplus R(\tilde{\nu})$  all constituents $[\nu]$ have degree \begin{align*} deg [\nu] & \leq deg(\tilde{\lambda}) + deg(\tilde{\mu}) \\ & = deg(\lambda^i) + n (-1) \lambda^i_n + deg(\mu^j) + n (-1) \mu^j_n \\ & < deg(\lambda) + n (-1) \lambda^i_n + deg(\mu) + n (-1) \mu^j_n \\ &= deg(\nu) +  n (-1) \lambda^i_n +  n (-1) \mu^j_n. \end{align*} Since $[\lambda^i] \otimes [\mu^j] = B^{- \lambda^i_n} \otimes B^{- \mu^j_n} ( [\tilde{\lambda}] \otimes [\tilde{\mu}])$, every constituent in $[\lambda^i] \otimes [\mu^j]$ has degree $\leq deg(\lambda^i) + deg(\mu^j) < deg(\lambda) + deg(\mu)$.
\end{proof}

\begin{example} $S^i \otimes S^j$ in $\calR_2$ with $i>j$: In this case one can show \cite{Heidersdorf-Weissauer-gl-2-2} \[ \A_{S^i} \otimes \A_{S^j} = \A_{S^{i+j}} + R(i+j-1,1) \oplus R(i+j-2,2) \oplus \ldots \oplus R(i,j) \oplus \tilde{R}\] where $\tilde{R}$ represents the summands with degree $ <  i+j$. All the classical solutions have length $\leq 2$, hence their highest weights occur in the $S^i \otimes S^j$ tensor product. The corresponding irreducible representations are \[S^{i+j}, BS^{i+j-2}, \ldots, B^j S^i.\] In fact \cite{Heidersdorf-Weissauer-gl-2-2} these constituents give half of the constituents in the middle Loewy layer of $M = S^i \otimes S^j$, the other half given by their twists with $B^{-1}$: \[ B^{-1}(S^{i+j} + BS^{i+j-2} + \ldots + B^j S^i).\]
\end{example}

%%%%%%%%%%%%%%%%%%%%%%%%%%%

\subsection{Projective covers} \label{proj-covers} Projective covers in a tensor product of irreducible representations or in iterated tensor products are particularly interesting. The main reason is that an indecomposable module in $\calR_n$ which is a direct summand in an iterated tensor product of irreducible representations is in the kernel of $DS$ if and only if its is projective.   \cite{Heidersdorf-Weissauer-Tannaka}. We show that we can exclude projective summands in $L(\lambda) \otimes L(\mu)$ if the weights are \textit{small}.

\begin{lem} If $P$ is a projective cover occurring as a direct summand in the decomposition $[\lambda^{\dagger}] \otimes [\mu^{\dagger}]$ with multiplicity $k$, then $R(\lambda) \otimes R(\mu)$ contains $P$ as a direct summand with multiplicity at least $k$.
\end{lem}

\begin{proof} We embed $[\lambda^{\dagger}]$ and $[\mu^{\dagger}]$ as the socles of the mixed tensors $R(\lambda)$ and $R(\mu)$. Projection of these modules on the top gives \begin{align*} \xymatrix{ 0 \ar[r] & ker(\varphi) \ar[r] & R(\lambda) \ar[r]^{\varphi} & [\lambda^{\dagger}] \ar[r] & 0 \\  0 \ar[r] & ker(\psi) \ar[r] & R(\mu) \ar[r]^{\psi} & [\mu^{\dagger}] \ar[r] & 0. } \end{align*} This gives the surjection $R(\lambda) \otimes R(\mu) \twoheadrightarrow [\lambda^{\dagger}] \otimes [\mu^{\dagger}]$. If $[\lambda^{\dagger}] \otimes [\mu^{\dagger}] = \bigoplus M_i \oplus \bigoplus P_i$ we get a surjection $R(\lambda) \otimes R(\mu) \twoheadrightarrow \bigoplus P_i$. Since the $P_i$ are projective this surjection has to split and hence the $\bigoplus P_i$ are direct summands in $R(\lambda) \otimes R(\mu)$. 
\end{proof}

This result implies that some tensor products $[\lambda^{\dagger}] \otimes [\mu^{\dagger}]$ do not have maximally atypical projective summands. Indeed if $|\lambda|+ |\mu| < n(n+1)/2$ or $l(\lambda) + l(\mu) <n$, no projective cover can occur in $R(\lambda) \otimes R(\mu)$ since the smallest degree of a partition defining a projective cover is $n(n+1)/2$ and the smallest length is $<n$ (the minimal projective cover is $R(n,n-1,\ldots,1) = P[n-1,n-2,\ldots,1,0]$). More generally the Loewy length of any subquotient of a module $M$ is smaller or equal to the one of $M$, hence we have \[ ll(R(\lambda) \otimes R(\mu) ) \geq ll(L(\lambda^{\dagger} \otimes L(\mu^{\dagger})).\] 

\begin{cor} If $l([\lambda]) + l([\mu]) < k$, the Loewy length of any maximal atypical direct summand is $<2k+1$. In particular if $l([\lambda]) + l([\mu]) < n$, $[\lambda] \otimes [\mu]$ does not have a maximal atypical projective summand.
\end{cor}

\begin{example} $S^i \otimes S^j$ does not contain any atypical projective summands. Indeed for $n=2$ this follows from \cite{Heidersdorf-Weissauer-gl-2-2}. For $n \geq 3$ none of the mixed tensors in the decomposition of $\A_{S^{i+1}} \otimes \A_{S^{j+1}}$ is projective. In fact one can show that $S^i \otimes S^j$ does not have maximal atypical summands of superdimension 0 \cite{Heidersdorf-Weissauer-Tannaka}.
\end{example}

\begin{lem} Suppose $\nu$ is a classical solution $c_{\lambda \mu}^{\nu} \neq 0$ and $l(\nu) \leq n$. If $[\nu]$ is a composition factor in an indecomposable projective module $P = R(\theta)$ occuring in $[\lambda] \otimes [\mu]$, then $[\nu] = A_{\theta}$.
\end{lem}

\begin{proof} By definition $(A_{\theta})_i \geq \nu_i$ for all $i \in \{1,\ldots,n\}$. In particular $[\theta]$ is positive. Hence $l(\theta) \leq n$. Hence \[ |\theta| = \sum_{i=1}^{n} (A_\theta)_i.\] Since $P=R(\theta)$ is a summand in $R(\lambda) \otimes R(\mu)$, $|\theta| \leq \nu|$. But if $[\nu] \neq A_{\theta}$, then $\sum_{i=1}^{n} (A_{\theta})_i > \sum_{i=1}^{n} \nu_i$, a contradiction.
\end{proof}

\begin{cor} If $\nu$ satisfies $c_{\lambda \mu}^{\nu} \neq 0$   and $l(\nu) \leq n$ and $d(\nu)<n$, $[\nu]$ is not a composition factor of a projective module $P$ in $[\lambda] \otimes [\mu]$.
\end{cor}

\begin{proof} By the last lemma we have $P = R(\nu)$. But $R(\nu)$ is projective if and only if $d(\nu) < n$.
\end{proof}

\begin{remark} For the significance of some of these estimates see \cite{Heidersdorf-Weissauer-gl-2-2} \cite{Heidersdorf-Weissauer-Tannaka}. Indeed if we know that a projective cover cannot occur in a tensor product, we can often exclude the existence of other maximal atypical direct summands of superdimension 0.
\end{remark}

%-------------------------------------------------------------

%-------------------------------------------------------------

\pagestyle{plain}
\phantomsection
\bibliographystyle{alpha}
\bibliography{main-mixed-tensors}{}

\end{document}